\documentclass[a4paper,11pt,leqno]{article}
\usepackage{amsmath,amsfonts,amsthm,amssymb}
\usepackage[left=3cm, right=3cm, top=3cm, bottom=3.5cm]{geometry}
\usepackage{graphicx}
\numberwithin{equation}{section}

\theoremstyle{plain}
\newtheorem{theorem}{Theorem}[section]

\newtheorem{lemma}[theorem]{Lemma}
\newtheorem{proposition}[theorem]{Proposition}

\newtheorem*{theoremA}{Theorem A}
\newtheorem*{theoremB}{Theorem B}

\theoremstyle{remark}

\theoremstyle{definition}

\newcommand{\R}{\mathbb{R}}
\newcommand{\SF}{\mathbb{S}}

\begin{document}

\title{\bf DENSITY ESTIMATES FOR VECTOR MINIMIZERS\\AND APPLICATIONS}

\author{Nicholas D.\ Alikakos\thanks{The first author was partially supported through the project PDEGE – Partial Differential Equations
Motivated by Geometric Evolution, co-financed by the European Union – European Social
Fund (ESF) and national resources, in the framework of the program Aristeia of the ‘Operational
Program Education and Lifelong Learning’ of the National Strategic Reference Framework
(NSRF).}\;\,\footnote{The research of N. Alikakos has been co-financed by the European Union – European Social
Fund (ESF) and Greek national funds through the �� Operational Program
 Education and Lifelong Learning’ of the National Strategic Reference
Framework (NSRF) - Research Funding Program: THALES}\hskip.2cm and Giorgio Fusco}

\date{}
\maketitle

\begin{abstract}
We extend the Caffarelli-Cordoba estimates to the vector case
in two ways, one of which has no scalar counterpart, and we give a few
applications for minimal solutions.
\end{abstract}
\section{Introduction}
This paper is concerned with solutions to the system
\begin{equation}\label{system}
\Delta u-W_u(u)=0,\quad u:D\rightarrow\R^m
\end{equation}
$D\subset\R^n$, where $D=\R^n$ is an important special case, $W:\R^m\rightarrow\R$, $W\geq 0$, with regularity specified later, and $W_u=(\frac{\partial W}{\partial u_1},\dots,\frac{\partial W}{\partial u_m})^\top$.

Unlike the scalar case $m=1$, where for a class of results, the form of the {\it potential} $W$ is immaterial, for the system the connectedness of $\{W=0\}\neq\emptyset$ plays a major role. Distinguished examples are:
(a) the phase transition model or vector Allen-Chan equation, where $W$ has a finite number $N$ of global minima $a_1,\dots,a_N$ (Baldo \cite{b}, Bronsard and Reitich \cite{br}),
(b) the Ginzburg-Landau system $\Delta u-(\vert u\vert^2-1)u=0$ (Bethuel, Brezis and Helein \cite{bbh}) and
(c) the phase separation system $\Delta u-\sum_{j\neq i}u_iu_j=0$ (Caffarelli and Lin \cite{cl}) and its variants.

System \ref{system} is the Euler-Lagrange equation for the functional
\begin{equation}\label{functional}
J_D(u)=\int_D\Big(\frac{1}{2}\vert\nabla u\vert^2
+W(u)\Big) dx.
\end{equation}
In the present paper we limit ourselves to uniformly bounded solutions to (\ref{system}) that are {\it minimal} in the sense that
\begin{equation}\label{min-def}
J_\Omega(u)=\min_v J_\Omega(v),\quad v=u\;\text{ on}\;\partial\Omega
\end{equation}
for every $\Omega$ open, bounded Lipschitz $\Omega\subset D$.

The basic estimate for such solutions is
\begin{equation}\label{basic-est}
\int_{B_B(x_0)}\Big(\frac{1}{2}\vert\nabla u\vert^2
+W(u)\Big) dx\leq C R^{n-1},
\end{equation}
$B_R(x_0)$ the $R$-ball in $\R^n$, center $x_0$, $B_R(x_0)\subset D$.

The key hypothesis in our theorems is
\begin{equation}\label{key-hyp}
W(a)=0,\;a\;\text{ isolated in }\;\{W=0\}
\end{equation}
This assumption excludes examples (b) and (c) above. It is well known that the phase transition model is linked to minimal surfaces ($m=1$) and Plateau Complexes ($m\geq 2$). In particular in the vector case entire solutions to (\ref{system}) are linked to singular minimal cones which unlike planes have additional hierarchical structure ( Alikakos \cite{a2} ).

The main purpose of this paper is the various extensions of the Caffarelli-Cordoba density estimates \cite{cc} to the vector case. In the scalar case, among other things, these estimates refine the linking of the phase transition model to minimal surfaces and have played a major role in the resolution of De Giorgi conjecture in higher dimensions ( Savin \cite{sa} ). Other extensions to the density estimates in different contexts have been provided by Farina and Valdinoci \cite{fv}, Savin and Valdinoci \cite{sv},\cite{sv2}, and Sire and Valdinoci \cite{siv}. Set
\begin{eqnarray}\label{0a-v-def}
\left\{\begin{array}{l}
A_R=\int_{B_R\cap\{\vert u-a\vert\leq\lambda\}}W(u) dx,\\\\
V_R=\mathcal{L}^n(B_R\cap\{\vert u-a\vert>\lambda\})
\end{array}\right.
\end{eqnarray}
where $\mathcal{L}^n$ stands for the $n$-dimensional Lebesgue measure. Note that $A_R$ satisfies $A_R\leq C R^{n-1}$ by (\ref{basic-est}). In the context of diffuse interfaces $A_R$ measures interface area while $V_R$ enclosed volume (\cite{cc}).
\begin{theoremA}\label{main-t}
Under (\ref{key-hyp}) and regularity of $W$ as in {\rm{(HA)}} in the next section, for $u:\R^n\rightarrow\R^m$, minimal, $\| u\|_{L^\infty}<\infty$, the following holds for $0<\lambda<\text{dist}(a,\{W=0\}\setminus\{a\})$:
\[\begin{split}
&\text{ If }\quad\quad\; V_1\geq\mu_0>0,\\\\
&\text{ then }\quad V_R\geq CR^n,\;R\geq 1,\;C=C(\mu_0, \lambda, \| u\|_{L^\infty}).
\end{split}\]
\end{theoremA}
\noindent The new points in the proof of Theorem A are the {\it polar form}
\begin{equation}\label{0polar}
\begin{split}
&u(x)=a+q^u(x)\nu^u(x),\\ & q^u(x)=\vert u(x)-a\vert,\quad \nu^u(x)=\frac{u(x)-a}{\vert u(x)-a\vert},
\end{split}
\end{equation}
the choice of the test functions which are limited to perturbations of the modulus $q^u$ and keep $\nu^u$ fixed,
\begin{equation}\label{0sigma-def}
\sigma=a+q^\sigma\nu^u,\quad q^\sigma=\min\{q^h,q^u\},
\end{equation}
and the resulting identity
\begin{equation}\label{0identity}
\begin{split}
&\frac{1}{2}\int_{B_R}(\vert\nabla q^u\vert^2-\vert\nabla q^\sigma\vert^2) dx\\&= J_{B_R}(u)-J_{B_R}(\sigma)+\frac{1}{2}\int_{B_R}\Big((q^\sigma)^2-(q^u)^2\Big)\vert\nabla\nu^u\vert^2 dx + \frac{1}{2}\int_{B_R}\Big(W(\sigma)-W(u)\Big) dx\\&\leq\frac{1}{2}\int_{B_R}\Big(W(\sigma)-W(u)\Big) dx
\end{split}
\end{equation}
where minimality on balls was used in the last inequality. The proof of Theorem A otherwise follows closely the argument in Caffarelli-Cordoba \cite {cc}.

We give a number of applications of Theorem A. We mention here a few and refer the reader to the main body of the paper for the precise statements.
\begin{enumerate}
\item \underline{Lower Bound}

For the phase transition model (a) above, under the hypotheses of Theorem A, and provided $u$ is not a constant, the lower bound holds
\begin{equation}\label{0lower-bound}
\int_{B_R(x_0)}\Big(\frac{1}{2}\vert\nabla u\vert^2
+W(u)\Big) dx\geq C R^{n-1},\quad R\geq R(x_0),
\end{equation}
$C>0$ independent of $x_0$.

We recall that for \underline{all} nonconstant solutions to (\ref{system}) and any $W\geq 0$ which allows $u\in W_{\rm loc}^{1,2}\cap L^\infty$ the estimate
\begin{equation}\label{0lower-bound-gen}
\int_{B_R(x_0)}\Big(\frac{1}{2}\vert\nabla u\vert^2
+W(u)\Big) dx\geq C R^{n-2}
\end{equation}
holds, and that (\ref{0lower-bound-gen}) can not in general be improved (Alikakos \cite{a}). In light of (\ref{basic-est}) estimate (\ref{0lower-bound}) is optimal.
\item \underline{Liouville-Rigidity Theorem}

If $u:\R^n\rightarrow\R^m$ is a bounded solution to (\ref{system}), minimal, and if either $\{W=0\}=\{a\}$, or $\inf_x d(u(x),\{W=0\}\setminus\{a\})>0,$ Then \[u\equiv a\]

This was proved in Fusco \cite{fu} with a different, though related method.
\item \underline{Linking}

For global minimizers of $J_\epsilon(u)=\int_D\Big(\frac{\epsilon^2}{2}\vert\nabla u\vert^2
+W(u)\Big) dx$, for $D$ open, bounded, with Dirichlet conditions on $\partial D$, and $W$ with exactly two minima, $W(a_1)=W(a_2)=0$, $W>0$ on $\R^m\setminus\{a_1,a_2\}$, $S_\epsilon=\{\vert u_\epsilon-a_j\vert=\gamma\}$, $\gamma\in(0,\vert a_1-a_2\vert)$  converges uniformly as $\epsilon\rightarrow 0^+$ to the minimal partition with Dirichlet conditions.

The proof is completely analogous to the corresponding scalar result in Caffarelli-Cordoba \cite{cc}.

Entire equivariant (minimal) solutions to (\ref{system}) correspond to minimal cones and possess a hierarchical structure at least for a class of symmetries. They were established by Bronsard, Gui and Schatzman \cite{bgs} for triple junctions, $n=m=2$, and by Gui and Schatzman \cite{gs} for quadruple junctions ($n=m=3$) and for general $n, m$ in a series of papers \cite{af2} \cite{a3} \cite{fu1}. In the papers \cite{bgs} \cite{gs} the hierarchical structure is built in, while in \cite{af2} \cite{a3} \cite{fu1} can be deduced a posteriori (see \cite{af4}).
\end{enumerate}

Our next theorem concerns an aspect that has no scalar counterpart. We look at the simplest possible set up for this kind of result. Consider (\ref{system}) in the class of symmetric solutions
\[u(\hat{x})=\hat{u}(x)\]
where for $z\in\R^d$ we denote by $\hat{z}$ the reflection of $z$ in the plane $\{z_1=0\}$,
\[\hat{z}=(-z_1,z_2\ldots,z_d),\]
and we take $W$ a $C^3$ potential, symmetric $W(u)=W(\hat{u}),\;u\in R^m$, and with exactly two minima $W(a_-)=W(a_+)=0$, $W>0$ on $\R^m\setminus\{a_+,a_-\}$. Under hypotheses of nondegeneracy for $a_+,a_-$  there is such a symmetric solution, minimal in the symmetric class, and satisfying the estimate
\[\vert u-a_+\vert+\vert\nabla u\vert\leq Ke^{-k x_1},\quad x_1\geq 0.\]
Consider the {\it Action}
\[A(v)=\int_\R\Big(\frac{1}{2}\vert v_s\vert^2
+W(v)\Big) ds\]
for symmetric $v\in W_{\mathrm loc}^{1,2}(\R;\R^m)\cap L^\infty(\R;\R^m)$, that connect at infinity the minima, $\lim_{s\rightarrow\pm\infty}v(s)=a_\pm$.

The key hypotheses in our theorems is that $A$ has a hyperbolic global minimum $e$ in the symmetric class. Following \cite{af4} we define the {\it Effective-Potential}
\begin{equation}\label{0eff-pot}
\mathcal{W}(v(\cdot))=A(v(\cdot))-A(e(\cdot))\geq 0
\end{equation}
and thus we have that
\begin{equation}\label{0eff-pot-e}
\mathcal{W}(e)=0,\;e\;\text{ isolated in }\;\{\mathcal{W}=0\}
\end{equation}
 (cfr. (\ref{key-hyp})) above.

The basic estimate in the present context is
\begin{equation}\label{0basic-est}
0\leq\int_{C_R(y_0)}\Big((\frac{1}{2}\vert\nabla u\vert^2
+W(u))-A(e)\Big) dx\leq C R^{n-2},
\end{equation}
$C_R(y_0)$ the cylinder $\R\times\mathcal{B}_R(y_0)$, $\mathcal{B}_R(y_0)$ the $R$-ball in $\R^{n-1}$ with center at $y_0\in\R^{n-1}$, $x=(s,y)$. Set
\[\|f\|=(\int_R\vert f(s)\vert^2 ds)^\frac{1}{2},\;\;f:\R\rightarrow\R^m\]
and by analogy to (\ref{0a-v-def})
\begin{eqnarray}\label{0a-v-def1}
\left\{\begin{array}{l}
\mathcal{A}_R=\int_{\mathcal{B}_R(y_0)\cap\{y: \| u(\cdot,y)-e(\cdot)\|\leq\lambda\}}\mathcal{W}(u) dy,\\\\
\mathcal{V}_R=\mathcal{L}^{n-1}(\mathcal{B}_R\cap\{y: \| u(\cdot,y)-e(\cdot)\|>\lambda\}).
\end{array}\right.
\end{eqnarray}
Note that $\mathcal{A}_R\leq C R^{n-2}$ by (\ref{0basic-est}).
\begin{theoremB}\label{main-fun}
Let $u$ symmetric, and minimal in the symmetry class, as above. Under {\rm(\ref{0eff-pot-e})} in the $\|\cdot\|$ sense, there is $\lambda^*>0$ such that, for $0<\lambda<\lambda^*$ the following holds:
\[\begin{split}
&\text{ If }\quad\quad\; \mathcal{V}_1\geq\mu_0>0,\\
&\text{ then }\quad \mathcal{V}_R\geq C R^{n-1},\;R\geq 1,\;C=C(\mu_0, \lambda, \| u\|_{L^\infty}).
\end{split}\]
\end{theoremB}
The proof of Theorem B, following \cite{af4}, implements the {\it polar form}
\[\begin{split}& u(\cdot,y)=e(\cdot)+q^u(y)\nu^u(\cdot,y),\\
& q^u(y)=\|u(\cdot,y)-e(\cdot)\|;\quad \nu^u(\cdot,y)=\frac{u(\cdot,y)-e(\cdot)}{\|u(\cdot,y)-e(\cdot)\|}
\end{split}\]
and utilizes test functions that vary only $q^u$,
\[\sigma(\cdot,y)=e(\cdot)+q^\sigma(y)\nu^u(\cdot,y),\quad q^\sigma=\min\{q^h,q^u\}\]
and employs the identity
 \begin{equation}\label{02young-1}
\begin{split}
&\frac{1}{2}\int_{\mathcal{B}_R}(\vert\nabla q^u\vert^2-\vert\nabla q^\sigma\vert^2) dy\\
&=J_{C_R}(u)-J_{C_R}(\sigma)
+\frac{1}{2}\int_{\mathcal{B}_R}\Big((q^\sigma)^2-(q^u)^2\Big)\sum_{i=1}^{n-1}\|\frac{\partial \nu^u}{\partial y_i}\|^2 dy+\int_{\mathcal{B}_R}(\mathcal{W}(\sigma)-\mathcal{W}(u))dy\\
&\leq\int_{\mathcal{B}_R}(\mathcal{W}(\sigma)-\mathcal{W}(u))dy.
\end{split}
\end{equation}
where in the last inequality minimality with respect to cylinders was used. Thus the proof, mutatis mutandis, follows Caffarelli-Cordoba \cite{cc}.

We now mention some of the applications of Theorem B and refer the reader to the main body of the paper for more information and precise statements.
\begin{enumerate}
\item Assume that the Action $A$ has exactly two global minima $e_-,e_+$, $\mathcal{W}(e_-)=\mathcal{W}(e_+)=0$, $\mathcal{W}>0$ otherwise, where $e_-,e_+$ satisfy the hypotheses of $e$ above. Assume for $u$ the hypotheses of Theorem B. Then for $0<\theta<\|e_--e_+\|$ the following is true:

    If
    \begin{equation}
\mathcal{L}^{n-1}(\mathcal{B}_1(y_0)\cap\{y:\| u(\cdot,y)-e_-(\cdot\,)\|\leq\theta\})\geq \mu_0>0\hskip1.5cm
\label{02mu0-cond-1}
\end{equation}
Then
\begin{equation}
\mathcal{L}^{n-1}(\mathcal{B}_R(y_0)\cap\{y:\| u(\cdot,y)-e_-(\cdot\,)\|\leq\theta\})\geq C R^{n-1}
\label{0mu-cond-1}
\end{equation}
 for $R\geq 1$,
 $C=C(\mu_0,\theta,\|u\|_{L^\infty})$, with a similar statement for $e_+$.
 \item Assume the hypothesis of Theorem B and suppose that either $\{\mathcal{W}=0\}=\{e\}$ or $\inf_y\|u(\cdot,y)-(\{\mathcal{W}=0\}\setminus\{e\})\|>0$, then
     \[u\equiv e.\]
     This was proved in \cite{af4} under the hypothesis $\{\mathcal{W}=0\}=\{e\}$ with a different though related method.
\end{enumerate}
We recall that Alama, Bronsard and Gui in \cite{abg} have established, under the hypothesis of (i) above, the existence of a solution $u:\R^2\rightarrow\R^2$ converging to $a_\pm$ as $s\rightarrow\pm\infty$, and converging to $e_\pm$ as $y\rightarrow\pm\infty$. Thus there are solutions genuinely higher dimensional connecting $e_+$ and $e_-$.
     The paper is structured as follows. In Part I Theorem A is stated and proved and its applications are presented in individual sections. Similarly in Part II Theorem B is stated and proved, followed by its applications.

\newpage

\centerline{\underline{\bf PART I}}
\section{Theorem A}
\subsection{Hypotheses and Statement}

\begin{description}
\item[({HA})] The potential $W:\R^m\rightarrow\R$ is nonnegative and $W(a)=0$ for some $a\in\R^m$.
Moreover $W\in C^\alpha(\R^m;\R)\cap C^1(\R^m\setminus\{a\};\R)$.

If $0<\alpha<2$ we assume
\begin{equation}
 W_u(a+\rho\nu)\cdot\nu\geq C^*\rho^{\alpha-1},\;\text{ for }\;0<\rho\leq\rho_0,\;\vert\nu\vert=1
\label{wu-lower-bound}
\end{equation}
where $\cdot$ denotes the Euclidean inner product in $\R^m$, $C^*$ a positive constant.

If $\alpha=2$ we assume, for some constant $C_0>0$,
\begin{equation}
 W_{uu}(a)\nu\cdot\nu\geq C_0>0,\;\text{ for }\;\vert\nu\vert=1.\hskip2cm
\label{wu-lower-bound-1}
\end{equation}
The figure 
below shows the behavior of $W$ for different values of $\alpha$.

\

\

\

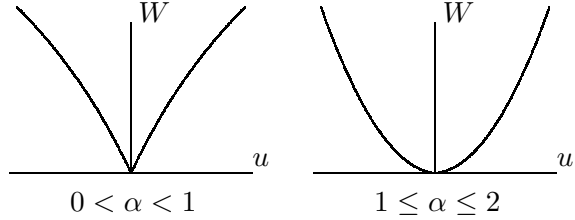
\begin{figure}[h]
\centering
\setlength{\unitlength}{.2mm}
\begin{picture}(200,50)(0,-170)
\put(-100,-150){\line(1,0){160}} \put(100,-150){\line(1,0){160}}
\put(-20,-150){\line(0,1){100}}\put(180,-150){\line(0,1){100}}
\put(-60,-175){$0<\alpha<1$}\put(140,-175){$1\leq\alpha\leq 2$}
\put(60,-145){$u$}\put(260,-145){$u$}
\put(-15,-50){$W$}\put(185,-50){$W$}
\qbezier(-20,-150)(10,-85)(55,-40)
\qbezier(-20,-150)(-50,-85)(-95,-40)

\qbezier(180,-150)(205,-148)(230,-100)
\qbezier(230,-100)(245,-70)(255,-40)

\qbezier(180,-150)(155,-148)(130,-100)
\qbezier(130,-100)(115,-70)(105,-40)

\end{picture}
\label{fig:1}
\end{figure}

\

\item[({HB})] $u:D\subset\R^n\rightarrow\R^m$, $u\in W_{\rm loc}^{1,2}(D;\R^m)\cap L^\infty(D;\R^m)$, is {\it minimal} in the sense that
\begin{equation}
J_{\Omega}(u)\leq J_{\Omega}(u+v),\;\text{ for }\;v\in W_0^{1,2}(\Omega;\R^m)
 \label{energy-inequality}
 \end{equation}
 for every open bounded set $\Omega\subset D$, where
 \begin{equation}
 J_\Omega(u)=\int_\Omega(\frac{1}{2}\vert\nabla u\vert^2+W(u))dx.
 \label{energy-def}
 \end{equation}

 \begin{equation}
 \vert u-a\vert<M,\;\;\vert\nabla u\vert<M,\;\text{ on }\;\R^n.
 \label{grad-bound}
 \end{equation}
\end{description}
\underline{Note}: In the proof of Theorem A we utilize minimality only on balls.

For each $z\in\R^k$, $k\geq 1$ and $r>0$ we let $B_r(z)\subset\R^k$ be the open ball of center $z$ and radius $r$ and $B_r$ the ball centered at the origin. We denote by $\mathcal{L}^k(E)$ the $k$-dimensional Lebesgue measure of a measurable set $E\subset\R^k$.

\begin{theoremA}\label{main}
Under hypothesis {\rm({HA})} and {\rm({HB})} above, for any $\mu_0>0$ and any $0<\lambda<d_0={\rm dist}(a,\{W=0\}\setminus\{a\})$, the condition
\begin{equation}
\mathcal{L}^n(B_1(x_0)\cap\{\vert u-a\vert>\lambda\})\geq\mu_0\hskip1.5cm,
\label{mu0-cond}
\end{equation}
provided $B_R(x_0)\subset\Omega$, implies the estimate
\begin{equation}
\mathcal{L}^n(B_R(x_0)\cap\{\vert u-a\vert>\lambda\})\geq C R^n,\;\text{ for }\;R\geq 1
\label{mu-cond}
\end{equation}
where $C=C(\mu_0,\lambda,M)$, $C$ independent of $x_0$ and independent of $u$.
\end{theoremA}
As in \cite{cc} Theorem A has the following important consequence
\begin{theorem}\label{main-1}
Assume there are $a_1\neq a_2\in\R^m$ such that
\[W(a_1)=W(a_2)=0,\;\;W(u)>0,\;\text{ for }\;u\not\in\{a_1,a_2\}\]
and assume that {\rm({HA})} holds at $a=a_j,\;j=1,2$. Let $u:\R^n\rightarrow\R^m$ is a minimizer in the sense of {\rm({HB})}. Then, given $0<\theta<\vert a_1-a_2\vert$ the condition
\begin{equation}\label{mu0-cond-1}
\mathcal{L}^n(B_1(x_0)\cap\{\vert u-a_1\vert\leq\theta\})\geq\mu_0>0
\end{equation}
implies the estimate
\begin{equation}\label{mu0-cond-2}
\mathcal{L}^n(B_R(x_0)\cap\{\vert u-a_1\vert\leq\theta\})\geq C R^n,\;\text{ for }\;R\geq 1
\end{equation}
where $C>0$ depends only on $\mu_0$, $\theta$ and $M$. An analogous statement applies to $a_2$.
\end{theorem}
\begin{proof}
Since $\vert u-a_1\vert\leq\theta$ implies $\vert u-a_2\vert>\vert a_1-a_2\vert-\theta=\lambda>0$ the assumption (\ref{mu0-cond-1}) implies
\[\mathcal{L}^n(B_1(x_0)\cap\{\vert u-a_2\vert>\lambda\})\geq\mu_0.\]
Therefore Theorem A yields
\[\mathcal{L}^n(B_R(x_0)\cap\{\vert u-a_2\vert>\lambda\})\geq C R^n,\;\text{ for }\;R\geq 1.\]
To conclude the proof we observe that
\[\{\vert u-a_2\vert>\lambda\}=\{\vert u-a_1\vert\leq\theta\}\cup(\{\vert u-a_1\vert>\theta\}\cap\{\vert u-a_2\vert>\lambda\})\]
and therefore $W(u)>0,\;\text{ for }\;u\not\in\{a_1,a_2\}$ and Lemma \ref{energy-upper-bound} below imply $\mathcal{L}^n(B_R(x_0)\cap(\{\vert u-a_1\vert>\theta\}\cap\{\vert u-a_2\vert>\lambda\})\leq C R^{n-1}$.
\end{proof}
\underline{Note:}\hskip.1cm  We note that the argument above when applied to potentials $W$ that vanish at more than two points: $W(a_1)=\cdots=W(a_N)=0$, $N\geq 3$ �, provides estimates (\ref{mu0-cond-2}) only for two of the minima, even if (\ref{mu0-cond-1}) holds for all $N$ of them. The selection of the particular two minima depends in general on $R$.
\subsection{The Proof of Theorem A}

 1.\underline{The Polar Form}

 We will utilize the {\it polar form} of a vector map $u\in W^{1,2}(A;\R^m)\cap L^\infty(A;\R^m)$, $A\subset\R^n$ open and bounded,
\begin{equation}
u(x)=a+q^u(x)\nu^u(x)
\label{polar-form}
\end{equation}
where
\begin{eqnarray}
q^u(x)=\vert u(x)-a\vert,\;\;\nu^u(x)=\left\{\begin{array}{l}
\frac{u(x)-a}{\vert u(x)-a\vert},\;\text{ if }\;u(x)\neq a,\\\\
0,\;\text{ if }\;u(x)=a.
\end{array}\right.
\label{q-and-nu}
\end{eqnarray}
We have \cite{afs} $q^u\in W^{1,2}(A)\cap L^\infty(A)$ and $\nabla\nu^u$ is measurable and such that $q^u\vert\nabla\nu^u\vert\in L^2(A)$ and
\begin{equation}
\int_A\vert\nabla u\vert^2dx=\int_A\vert\nabla q^u\vert^2dx+\int_A(q^u)^2\vert\nabla\nu^u\vert^2dx.
\label{kinetic-decomposition}
\end{equation}
Moreover for $q^h\in W^{1,2}(A)\cap L^\infty(A)$, $q^u\geq 0$
the vector function $\sigma$ defined via
\begin{equation}
\sigma=a+q^\sigma\nu^u,\;\;q^\sigma=\min\{q^h,q^u\}
\label{sigma}
\end{equation}
is in $W^{1,2}(A;\R^m)\cap L^\infty(A;\R^m)$ and satisfies the corresponding (\ref{kinetic-decomposition}).
\vskip.3cm
By the polar form (\ref{kinetic-decomposition}) of the energy and the minimality of $u$ assumed in ({HB}) it follows that
\begin{equation}\label{young-1}
\begin{split}
&\frac{1}{2}\int_{B_R}(\vert\nabla q^u\vert^2-\vert\nabla q^\sigma\vert^2) dx\\
&=J_{B_R}(u)-J_{B_R}(\sigma)
+\frac{1}{2}\int_{B_R}\Big((q^\sigma)^2-(q^u)^2\Big)\vert\nabla\nu^u\vert^2dx+\int_{B_R}(W(\sigma)-W(u))dx\\
&\leq\int_{B_R}(W(\sigma)-W(u))dx
\end{split}
\end{equation}
where we have also used the definition (\ref{sigma}) of $\sigma$ which implies $q^\sigma\leq q^u$.

\noindent 2.\underline{The Isoperimetric Inequality for Minimizers}

  We will assume that $q^h\geq q^u$ on $\partial B_R$ and therefore by (\ref{sigma}) that $q^\sigma=q^u$ on $\partial B_R$, $q^h$ to be further specified later. Define
\begin{eqnarray}
\left\{\begin{array}{l}
A_r=\int_{B_r\cap\{q^u\leq\lambda\}}W(u)dx,\\\\
V_r=\mathcal{L}^n(B_r\cap\{q^u>\lambda\}).
\end{array}\right.
\label{a-v-def}
\end{eqnarray}
We also define the cut-off function
\begin{equation}
\beta=\min\{q^u-q^\sigma,\lambda\},\;\text{ on }\;B_R,\;\lambda>0\;\text{ small. }
\label{cut-off}
\end{equation}
which is related via the map $a+\beta\nu^u$ to the variation $\sigma$ in (\ref{sigma}).
The modification in the definition of $A$ with the integration over the sub-level set together with the definition of the function $\beta$ in the context of the Caffarelli-Cordoba \cite{cc} set-up was introduced in Valdinoci \cite{v}.
By applying the inequality in \cite{eg} pag.141 to $\beta^2$ we obtain
\begin{equation}
\begin{split}
\Big(\int_{B_R}\beta^\frac{2n}{n-1}dx\Big)^\frac{n-1}{n}&=\Big(\int_{B_R}(\beta^2)^\frac{n}{n-1}dx\Big)^\frac{n-1}{n}dx\\&\leq C\int_{B_R}\vert\nabla(\beta^2)\vert dx\leq 2 C\int_{B_R\cap\{q^u-q^\sigma\leq\lambda\}}\vert\nabla\beta\vert\vert\beta\vert dx,
\end{split}\label{iso-inequality}
\end{equation}
where $C>0$ is a constant independent of $R$ and we have used $\beta=0$ on $\partial B_R$ and the fact that $\nabla\beta=0$ a.e. on $q^u-q^\sigma>\lambda$. By Young's inequality, for $A>0$ we have
\begin{equation}\label{young-0}
\begin{split}
&\Big(\int_{B_R}\beta^\frac{2n}{n-1}dx\Big)^\frac{n-1}{n}\leq 2 C\int_{B_R\cap\{q^u-q^\sigma\leq\lambda\}}\vert\nabla\beta\vert\vert\beta\vert dx\\&\leq
 C A\int_{B_R\cap\{q^u-q^\sigma\leq\lambda\}}\vert\nabla\beta\vert^2 dx+
\frac{C}{A}\int_{B_R\cap\{q^u-q^\sigma\leq\lambda\}}\beta^2 dx
\\&\leq
 C A\int_{B_R}\vert\nabla(q^u-q^\sigma)\vert^2 dx+
\frac{C}{A}\int_{B_R\cap\{q^u-q^\sigma\leq\lambda\}}(q^u-q^\sigma)^2 dx
\\&=
 C A\Big(\int_{B_R}(\vert\nabla q^u\vert^2-\vert\nabla q^\sigma\vert^2) dx-2\int_{B_R}\nabla q^\sigma\cdot\nabla(q^u-q^\sigma) dx\Big)\\&+
\frac{C}{A}\int_{B_R\cap\{q^u-q^\sigma\leq\lambda\}}(q^u-q^\sigma)^2 dx.
\end{split}
\end{equation}

From (\ref{young-0}) and (\ref{young-1}) it follows
\begin{equation}
\begin{split}
&\Big(\int_{B_R}\beta^\frac{2n}{n-1}dx\Big)^\frac{n-1}{n}\leq\\
 &\leq
 2C A\Big(\int_{B_R}(W(\sigma)-W(u))dx-\int_{B_R}\nabla q^\sigma\cdot\nabla(q^u-q^\sigma) dx\Big)\\&+
\frac{C}{A}\int_{B_R\cap\{q^u-q^\sigma\leq\lambda\}}(q^u-q^\sigma)^2 dx.
\end{split}
\label{young}
\end{equation}

\noindent 3.\underline{The case $0<\alpha<2$.}

Assume that
 $q^h\in W^{1,2}(B_R)\cap L^\infty(B_R)$ satisfies
\begin{equation}
q^h=
 0,\;\text{ on }\;B_{R-T}\;\text{ for some fixed }\; T>0.
\label{h-properties}
\end{equation}
\underline{{\it The Lower Bound}}

From (\ref{h-properties}) it follows
\begin{equation}
\Big(\int_{B_R}\beta^\frac{2n}{n-1}dx\Big)^\frac{n-1}{n}
\geq\Big(\int_{B_{R-T}\cap\{q^u>\lambda\}}\beta^\frac{2n}{n-1}\Big)^\frac{n-1}{n} dx\geq\lambda^2\mathcal{L}^n(B_{R-T}\cap\{q^u>\lambda\})^\frac{n-1}{n}
\label{lower-b}
\end{equation}
where we have also used (\ref{h-properties}) which implies $q^\sigma=0$ on $B_{R-T}$.
\vskip.3cm
\noindent\underline{{\it The Upper Bound}}

The objective is to estimate the right and side of (\ref{young}) by the first term involving the potential. Naturally the third term can be handled more easily for $\alpha<2$. For handling the second term one needs a very particular choice of $q^h$. The splitting of the integrations over $B_{R-T}$  and the rest aims at deriving a difference inequality involving the quantities in (\ref{a-v-def}), as in (\ref{caff-cord}). A major difference between $\alpha<2$ and $\alpha=2$ is in the choice of $q^h$, that can vanish on $B_{R-T}$ for $\alpha<2$, while can only be exponentially small (in $T$) for $\alpha=2$.

\noindent We begin with $B_{R-T}$.

\noindent Since $q^\sigma=0$ on $B_{R-T}$ the right hand side $I$ of (\ref{young}) on $B_{R-T}$ reduces to
\begin{equation}\label{rhs-br-t}
\begin{split}
& I=-2C A\int_{B_{R-T}}W(u)dx+
\frac{C}{A}\int_{B_{R-T}\cap\{q^u\leq\lambda\}}(q^u)^2 dx\\
&\leq-2C A\int_{B_{R-T}\cap\{q^u\leq\lambda\}}W(u)dx+
\frac{C}{A}\int_{B_{R-T}\cap\{q^u\leq\lambda\}}(q^u)^2 dx
\end{split}
\end{equation}
\noindent\underline{{\it Claim }1}

{\it Assume $\lambda\leq\rho_0$, $\rho_0$ the constant in {\rm({HA})}. Then there exists $A_0>0$ independent of $R$ such that
\begin{equation}\label{i-estimate}
I\leq-C A\int_{B_{R-T}\cap\{q^u\leq\lambda\}}W(u)dx,\;\text{ for }\;A>A_0.
\end{equation}}

\begin{proof}
From ({HA}) $q^u\leq\lambda\leq\rho_0$ it follows
\[\begin{split}
& W(u)=\int_0^{q^u} W_u(a+s\nu^u)\cdot\nu^u ds\geq\frac{C^*}{\alpha}(q^u)^\alpha\\
&\text{hence}\hskip.2cm -A W(u)+\frac{1}{A}(q^u)^2\leq(q^u)^\alpha(\frac{-A C^*}{\alpha}+\frac{\lambda^{2-\alpha}}{A})
\end{split}\]
and therefore for $A>\sqrt{\alpha\lambda^{2-\alpha}/C^*}$ we obtain
\[-C A\int_{B_{R-T}\cap\{q^u\leq\lambda\}}W(u)dx+
\frac{C}{A}\int_{B_{R-T}\cap\{q^u\leq\lambda\}}(q^u)^2 dx\leq 0.\]
This and (\ref{rhs-br-t}) conclude the proof of the claim.
\end{proof}
Next we consider the right hand side of (\ref{young}) on $B_R\setminus B_{R-T}$.

\noindent Set
\[\begin{split}
& I_1=2C A\int_{B_R\setminus B_{R-T}}(W(\sigma)-W(u))dx+
\frac{C}{A}\int_{(B_R\setminus B_{R-T})\cap\{q^u-q^\sigma\leq\lambda\}}(q^u-q^\sigma)^2 dx,\\
& I_2=-2C A\int_{B_R\setminus B_{R-T}}\nabla q^\sigma\cdot\nabla(q^u-q^\sigma) dx.
\end{split}\]
\noindent\underline{{\it Claim }2}

{\it Assume $\lambda\leq\min\{\rho_0,1\}$. Then there exists constant $\tilde{C}>0$ independent of $R$ such that
\begin{equation}\label{i1-estimate}
I_1\leq\tilde{C} A\mathcal{L}^n((B_R\setminus B_{R-T})\cap\{q^u>\lambda\})+\frac{\tilde{C}}{A}\int_{(B_R\setminus B_{R-T})\cap\{q^u\leq\lambda\}}W(u) dx,\;\text{ for }\;A>0.
\end{equation}}
\begin{proof}
We split the integration in $B_R\setminus B_{R-T}$ over $\{q^u\leq\lambda\}$ and $\{q^u>\lambda\}$.
From $q^\sigma\leq q^u$, $q^u\leq\lambda\leq\rho_0$ we have
\[\int_{(B_R\setminus B_{R-T})\cap\{q^u\leq\lambda\}}(W(\sigma)-W(u)) dx\leq 0\]
and therefore from (\ref{grad-bound}) it follows
\begin{equation}\label{wsigma-wu}
\int_{B_R\setminus B_{R-T}}(W(\sigma)-W(u))dx\leq W_M\mathcal{L}^n((B_R\setminus B_{R-T})\cap\{q^u>\lambda\})
\end{equation}
where $W_M=\max_{\vert u-a\vert\leq M}W(u)$.
As in the proof of Claim 1, for $q^\sigma\leq q^u\leq\lambda\leq\min\{\rho_0,1\}$, we get
\[W(u)\geq\frac{C^*}{\alpha}(q^u)^\alpha\geq\frac{C^*}{\alpha}(q^u-q^\sigma)^\alpha\geq\frac{C^*}{\alpha}(q^u-q^\sigma)^2\]
which implies
\[\int_{(B_R\setminus B_{R-T})\cap\{q^u\leq\lambda\}}(q^u-q^\sigma)^2 dx\leq\frac{\alpha}{C^*}\int_{(B_R\setminus B_{R-T})\cap\{q^u\leq\lambda\}}W(u)dx.\]
This and (\ref{wsigma-wu}) establish Claim 2 with $\tilde{C}=C\max\{\alpha/C^*,2W_M, M^2\}$.
\end{proof}
We now complete the definition (\ref{h-properties}) of $q^h$ by setting as in \cite{cc}
\begin{equation}\label{h-def}
q^h(x)=H(\vert x\vert-(R-T))^\frac{2}{2-\tau},\;\text{ on }\;B_R\setminus B_{R-T}
\end{equation}
where $\tau=\max\{\alpha,1\}$ and $H=M/T^\frac{2}{2-\tau}$ is chosen so that $q^h=M$ on $\partial B_R$. Note that $q^h$ is $C^1$ on $B_R$ and
\begin{equation}\label{h-properties-1}
\nabla q^h=\nabla q^\sigma=0\;\text{ on }\;\partial B_{R-T}
\end{equation}
where we have also used that $q^\sigma\leq q^h$. The function $[0,T]\ni s\mapsto q(s)=H s^\frac{2}{2-\tau}$ satisfies
\begin{equation}\label{q-laplacian}
q^{\prime\prime}=c_Hq^{\tau-1},\;\;q^\prime=\sqrt{2c_H/\tau}q^\frac{\tau}{2}
\end{equation}
where $c_H$ is a constant that depends on $H$. Since $\tau<2$ implies $\tau-1<\frac{\tau}{2}$, (\ref{q-laplacian}) yields
\begin{equation}\label{q-laplacian-est}
\Delta q^h\leq C_1(q^h)^{\tau-1}
\end{equation}
with $C_1>0$ independent of $R$.
\vskip.3cm
\noindent\underline{{\it Claim }3}

{\it There exists $\hat{C}>0$ independent of $R$ such that
\begin{equation}\label{i2-estimate}
\begin{split}
& I_2\leq\hat{C} A\mathcal{L}^n((B_R\setminus B_{R-T})\cap\{q^u>\lambda\})\\
&+\hat{C} A\int_{(B_R\setminus B_{R-T})\cap\{q^u\leq\lambda\}}W(u) dx,\;\text{ for }\;A>0.
\end{split}
\end{equation}
}
\begin{proof}
From (\ref{h-properties-1}) and $q^u=q^\sigma$ on $\partial B_R$ and integration by parts it follows
\begin{equation}\label{by-parts}
I_2=2C A\int_{(B_R\setminus B_{R-T})}\Delta q^\sigma(q^u-q^\sigma)dx
=2C A\int_{(B_R\setminus B_{R-T})\cap\{q^h<q^u\}}\Delta q^h(q^u-q^h)dx
\end{equation}
where we have observed that $q^u=q^\sigma$ on the set $\{q^h\geq q^u\}$ and that $q^\sigma=q^h$ on the set $\{q^h<q^u\}$.
From (\ref{q-laplacian-est}) and $q^h\leq q^u$ it follows
\begin{equation}\label{i2-estimate-1}
\begin{split}
I_2&\leq 2C C_1 A\int_{B_R\setminus B_{R-T}\cap\{q^h<q^u\}} (q^h)^{\tau-1}(q^u-q^h)dx\\
&\leq 2C C_1 A\int_{B_R\setminus B_{R-T}\cap\{q^h<q^u\}} (q^u)^\tau dx.
\end{split}
\end{equation}
As before we split the integration over $\{q^u\leq\lambda\}$ and $\{q^u>\lambda\}$.
To conclude the proof we observe that $\lambda\leq\min\{\rho_0,1\}$ and ({HA}) imply
\[(q^u)^\tau\leq(q^u)^\alpha\leq\frac{\alpha}{C^*}W(u),\;\text{ on }\;\{q^u\leq\lambda\}\]
while (\ref{grad-bound}) implies
\[(q^u)^\tau\leq M^\tau,\;\text{ on }\;\{q^u>\lambda\}.\]
\end{proof}
\noindent We are now in the position of completing the proof of Theorem A for the case $0<\alpha<2$. By recalling the definition  of $A_R$ and $V_R$ in (\ref{a-v-def}) and by collecting all the estimates (\ref{lower-b}),(\ref{i-estimate}),(\ref{i1-estimate}) and (\ref{i2-estimate}) we have for fixed $A>A_0$
\[
\lambda^2(V_{R-T})^\frac{n-1}{n}+CA\,A_{R-T}\leq(\tilde{C}+
\hat{C})A\Big(V_R-V_{R-T}\Big)
+(\frac{\tilde{C}}{A}+\hat{C}A)\Big(A_R-A_{R-T}\Big)\]
and consequently
\begin{equation}\label{caff-cord}
C(\lambda)\Big((V_{R-T})^\frac{n-1}{n}+V_{R-T}\Big)\leq(V_R-V_{R-T})
+(A_R-A_{R-T})
\end{equation}
with $C(\lambda)=\frac{\min\{\lambda^2,CA\}}{\max\{(\tilde{C}+\hat{C})A,\frac{\tilde{C}}{A}+\hat{C}A\}}$. Equation (\ref{caff-cord}) is exactly the difference scheme in \cite{cc}. Therefore as in \cite{cc}, using also the assumption (\ref{mu0-cond}), we deduce that there are $C(\lambda,\mu_0)>0$ and $k_0\geq 1$ such that
\begin{equation}\label{caff-cord-1}
V_{kT}+A_{kT}\geq C(\lambda,\mu_0)k^n,\;\text{ for }\;k\geq k_0.
\end{equation}
To complete the argument we recall the basic estimate (\ref{energy-upper-bound-1}) below (c.f. Lemma 1 in \cite{cc} for the scalar case. The proof is similar for the vector case)
\begin{lemma}\label{energy-upper-bound}
Assume that $W$ satisfies {\rm({HA})} and assume that $u$ is minimal as defined in {\rm({HB})}. Then there is a constant $C>0$, depending on $M$, independent of $\xi$ and such that
\begin{equation}\label{energy-upper-bound-1}
\int_{B_R(\xi)}\Big(\frac{1}{2}\vert\nabla u\vert^2+W(u)\Big)dx\leq C R^{n-1},\;\text{ for }\;R>0.\end{equation}
\end{lemma}
From (\ref{energy-upper-bound-1}) we obtain $A_{kT}\leq C(kT)^{n-1}$. This concludes the proof of Theorem A in the case $0<\alpha<2$ for $\lambda>0$ small. The restriction on the smallness of $\lambda$ is easily removed via (\ref{energy-upper-bound-1}).
\vskip.3cm
\noindent 4.\underline{The case $\alpha=2$.}

 We let $\varphi:B_R\rightarrow\R$ the solution of the problem
 \begin{eqnarray}\label{phi}
 \left\{\begin{array}{l}\Delta\varphi=c_1\varphi,\;\text{ on }\;B_R,\\
 \varphi=1,\;\text{ on }\;\partial B_R,\end{array}\right.
 \end{eqnarray}
 where $c_1<c_0$ will be chosen later and $c_0$ is the constant in ({HA}). It is well known that $\varphi$ satisfies the exponential estimate
 \begin{equation}\label{exp-phi}
 \varphi(R-r)\leq e^{-c_2(R-r)},\;\text{ for }\;r\in[0,R],\;R\geq 1,
 \end{equation}
 for some $c_2>0$.

 Define
 \begin{equation}\label{h-sigma-def}
 \begin{split}
 q^h=\varphi M,\hskip.8cm
 \end{split}
 \end{equation}
  and as before
  \begin{equation}\label{h-sigma-def-1}
 \begin{split}
 & q^\sigma=\min\{q^u,q^h\},\\
 & \beta=\min\{q^u-q^\sigma,\lambda\},
 \end{split}
 \end{equation}
 From (\ref{young}), $q^\sigma=q^u$ on $\partial B_R$, and an integration by parts we get
  \begin{equation}\label{sobolev-1}
  \begin{split}
 &(\int_{B_R}\beta^\frac{2n}{n-1})^\frac{n-1}{n}\\
 &\leq 2CA\int_{B_R}\Big(W(\sigma)-W(u)+\Delta q^\sigma(q^u- q^\sigma)\Big)dx +\frac{C}{A}\int_{B_R\cap\{q^u-q^\sigma<\lambda\}}(q^u-q^\sigma)^2 dx\\
 &=2CA\int_{B_R\cap\{q^u>q^h\}}\Big(W(h)-W(u)+\Delta q^h(q^u-q^h)\Big)dx +\frac{C}{A}\int_{B_R\cap\{0<q^u-q^h<\lambda\}}(q^u-q^h)^2 dx,
 \end{split}
 \end{equation}
 where we have used that $q^u>q^\sigma$ implies $q^\sigma=q^h$, $h=a+q^h\nu^u$.
By ({HA}) there is $\lambda^*>\lambda$ sufficiently small (and fixed from now on)  so that the maps $s\mapsto W(a+s\nu)$ and $s\mapsto W_u(a+s\nu)\cdot\nu$ are increasing in $[0,\lambda^*]$.
\vskip.3cm
\noindent\underline{{\it Claim }4}
 {\it \begin{equation}\label{sobolev-2}
  \begin{split}
 &(\int_{B_R}\beta^\frac{2n}{n-1}dx)^\frac{n-1}{n}\\
 &\leq 2CA\int_{B_R\cap\{q^u>q^h\}\cap\{q^u>\lambda^*\}}\Big(W(h)-W(u)+\Delta q^h(q^u-q^h)\Big)dx\\ &\hskip6cm+\frac{C}{A}\int_{B_R\cap\{0<q^u-q^h<\lambda\}\cap\{q^u>\lambda^*\}}(q^u-q^h)^2dx.
 \end{split}
 \end{equation}}

 \begin{proof}
In $B_R\cap\{q^u\leq\lambda^*\}$ we have
\begin{equation}\label{wu-wh}
\begin{split}
& W(u)-W(h)=\int_{q^h}^{q^u}W_u(a+s\nu)\cdot\nu ds\geq\int_{q^h}^{q^u}c_0s ds=\frac{1}{2}c_0(q^u+q^h)(q^u-q^h),\\
& \Delta q^h(q^u-q^h)=c_1q^h(q^u-q^h),
\end{split}
\end{equation}
where we have also utilized (\ref{phi}), (\ref{h-sigma-def}) and (\ref{h-sigma-def-1}). From (\ref{wu-wh}) it follows

\begin{equation}\label{integrand}
\begin{split}
& 2CA(W(h)-W(u)+\Delta q^h(q^u- q^h))+\frac{C}{A}(q^u-q^h)^2\\
&\leq \big(2CA(-\frac{1}{2}c_0(q^u+q^h)+c_1q^h)+\frac{C}{A}(q^u-q^h)\big)(q^u-q^h).
\end{split}
\end{equation}
For $c_1>0$ small and $A>0$ large ($c_1\leq\frac{1}{4}c_0$ and $A\geq\sqrt{\frac{2}{c_0}}$) the last expression in (\ref{integrand}) is negative. Therefore we also have
\begin{equation}\label{lambdastar}
\begin{split}
& 2CA\int_{B_R\cap\{q^u>q^h\}\cap\{q^u\leq\lambda^*\}}\Big(W(h)-W(u)+\Delta q^h(q^u- q^h)\Big)dx\\ &\hskip3cm+\frac{C}{A}\int_{B_R\cap\{0<q^u-q^h<\lambda\}\cap\{q^u\leq\lambda^*\}}(q^u-q^h)^2dx\leq 0.
\end{split}
\end{equation}
 This and (\ref{sobolev-1}) conclude the proof of Claim 4.
\end{proof}
 Set $R=(k+1)T$ where $T>0$ is a large number to be chosen later.
 Set
  \begin{equation}\label{omega}
 \omega_j=\mathcal{L}^n((B_{jT}\setminus B_{(j-1)T})\cap\{q^u>\lambda^*\}),\;j=1,\ldots,k+1.
 \end{equation}
 \noindent\underline{{\it Claim }5}
{\it \begin{equation}\label{basic-inequality}
C_0(\sum_{j=1}^k\omega_j)^\frac{n-1}{n}\leq \sum_{j=1}^k e^{-c_2jT}\omega_{k+1-j}+\omega_{k+1},\;k=1,\ldots
\end{equation}
where $c_2$ is the constant in (\ref{exp-phi}) and $C_0>0$ is a constant, $C_0=C_0(A,\lambda,M)$.
}
 \begin{proof}
 On $B_{kT}$ we have $q^h\leq M e^{-c_2 T}$ and therefore we can choose $T>0$ so large that
 \begin{equation}\label{0-on-bk}
 \begin{split}
 & x\in B_{kT}\cap\{q^u>\lambda^*\}\;\Rightarrow\;q^u-q^h\geq\lambda^*-Me^{-c_2 T}>\lambda\\
 &\Rightarrow\;B_{kT}\cap\{q^u-q^h<\lambda\}\cap\{q^u>\lambda^*\}=\emptyset.
 \end{split}
 \end{equation}
 We begin by estimating part of the right hand side of (\ref{sobolev-2}) over $B_R\setminus B_{R-T}$ by utilizing (\ref{0-on-bk}) and (\ref{i2-estimate-1})
 \begin{equation}\label{on-bk1}
 \begin{split}
 & 2CA\int_{(B_{(k+1)T}\setminus B_{kT})\cap\{q^u>q^h\}\cap\{q^u>\lambda^*\}}\Big(W(h)-W(u)+\Delta q^h(q^u-q^h)\Big)dx\\ &\hskip3cm+\frac{C}{A}\int_{B_{(k+1)T}\cap\{0<q^u-q^h<\lambda\}\cap\{q^u>\lambda^*\}}(q^u-q^h)^2dx\\
 &\leq 2CA\int_{(B_{(k+1)T}\setminus B_{kT})\cap\{q^u>q^h\}\cap\{q^u>\lambda^*\}}\Big(W(h)+\Delta q^h(q^u-q^h))\Big)dx\\ &\hskip3cm+\frac{C}{A}\int_{(B_{(k+1)T}\setminus B_{kT})\cap\{0<q^u-q^h<\lambda\}\cap\{q^u>\lambda^*\}}(q^u-q^h)^2dx\\
 &\leq \big(2CA(\overline{W}+c_1M^2)+\frac{C}{A}\lambda^2\big)\mathcal{L}^n((B_{(k+1)T}\setminus B_{kT})\cap\{q^u>\lambda^*\})\\
 &=C^*\omega_{k+1}
 \end{split}
 \end{equation}
 where we have set $\overline{W}=\max_{\vert u-a\vert\leq M} W(u)$ and $C^*=2CA(\overline{W}+c_1M^2)+\frac{C}{A}\lambda^2$.

 Next we estimate the remaining part of (\ref{sobolev-2}) over $B_{R-T}$. The smoothness of $W$ implies that there are $C_0>0$ and  $\bar{q}>0$ such that
\begin{equation}\label{w-upper-bound}
W(a+s\nu)\leq\frac{1}{2}C_0s^2,\;\text{ for }\;s\in[0,\bar{q}].
\end{equation}
We can assume $T>0$ so large that $M e^{-c_2T}\leq\bar{q}$. Then we have
\begin{equation}\label{j-bound}
\begin{split}
& x\in((B_{(k+1-j)T}\setminus B_{(k-j)T})\cap\{q^u>\lambda^*\}\cap\{q^u>q^h\}\\
& \Rightarrow\; W(h)+\Delta q^h(q^u-q^h)\leq M^2 e^{-c_2jT}(\frac{1}{2}C_0e^{-c_2jT}+c_1)\\
& \Rightarrow\;2CA\int_{(B_{(k+1-j)T}\setminus B_{(k-j)T})\cap\{q^u>\lambda^*\}\cap\{q^u>q^h\}}(W(h)+\Delta q^h(q^u-q^h))dx\\
& \leq 2CAM^2(\frac{1}{2}C_0+c_1) e^{-c_2jT}\omega_{k+1-j)}=C^\circ\omega_{k+1-j}\epsilon^j
\end{split}
\end{equation}
where we have set $C^\circ=2CAM^2(\frac{1}{2}C_0+c_1)$ and $\epsilon=e^{-c_2T}$.
From (\ref{j-bound}) we obtain
\begin{equation}\label{sum-j}
2CA\int_{B_{kT}\cap\{q^u>\lambda^*\}\cap\{q^u>q^h\}}(W(h)-W(u)+\Delta q^h(q^u-q^h)) dx\leq C^\circ\sum_{j=1}^k\epsilon^j\omega_{k+1-j}.
\end{equation}
Combining (\ref{sum-j}), (\ref{on-bk1}) in (\ref{sobolev-2}) we obtain the upper bound
\begin{equation}\label{sobolev-3}
(\int_{B_{(k+1)T}}\beta^\frac{2n}{n-1}dx)^\frac{n-1}{n}\leq C^\circ\sum_{j=1}^k\epsilon^j\omega_{k+1-j}+C^*\omega_{k+1}
\end{equation}
To estimate the left hand side of (\ref{sobolev-3}) from below we observe that (\ref{0-on-bk}) implies
\begin{equation}\label{first-rhs-est}
\begin{split}
&(\int_{B_{kT}\cap\{q^u-q^h<\lambda\}\cap\{q^u>\lambda^*\}}(q^u-q^h)^\frac{2n}{n-1}dx
+\int_{B_{kT}\cap\{q^u-q^h\geq\lambda\}\cap\{q^u>\lambda^*\}}\lambda^\frac{2n}{n-1}dx)^\frac{n-1}{n}\\
&=
(\int_{B_{kT}\cap\{q^u>\lambda^*\}}\lambda^\frac{2n}{n-1}dx)^\frac{n-1}{n}\\
&=\lambda^2(\sum_{j=1}^k\omega_j)^\frac{n-1}{n}\\
&=(\int_{B_{kT}\cap\{q^u>\lambda^*\}}(\beta^2)^\frac{n}{n-1}dx)^\frac{n-1}{n}\\
&\leq(\int_{B_{(k+1)T}}\beta^\frac{2n}{n-1}dx)^\frac{n-1}{n}.
\end{split}
\end{equation}
Combining this with (\ref{sobolev-3}) we obtain (\ref{basic-inequality}). The proof of Claim 5 is complete

\end{proof}
 \noindent\underline{{\it Claim }6}

{\it From
{\rm (\ref{basic-inequality})} it follows
\begin{equation}\label{claim}
\omega_k\geq c^*k^{n-1},\;\text{ for }\;k=1,2,\ldots.
\end{equation}
for some $c^*>0$. Then {\rm(\ref{claim})} implies
\[\mathcal{L}^n(B_R\cap\{q^u>\lambda^*\})\geq\frac{c^*}{n2^nT^n}R^n,\;\text{ for }\;R\geq T.\]
 }
 \begin{proof}
 We proceed by induction. For $k=1$ (\ref{claim}) holds by (\ref{mu0-cond}) for any $0<c^*\leq\mu_0$, $T\geq 1$. Thus we assume that (\ref{claim}) holds true for $j\leq k$ and show that it is true for $k+1$. From the inductive assumption
we have
\begin{equation}\label{left-inductive}
\frac{c^*}{n}k^n=c^*\int_0^kj^{n-1}dj\leq c^*\sum_{j=1}^kj^{n-1}\leq\sum_{j=1}^k\omega_j.
\end{equation}
Therefore for the left hand side of (\ref{basic-inequality}) we have the lower bound
\begin{equation}\label{left-inductive-1}
\frac{C_0}{2^n}(\frac{c^*}{n})^\frac{n-1}{n}(k+1)^{n-1}\leq C_0(\frac{c^*}{n})^\frac{n-1}{n}k^{n-1}.
\end{equation}
Observe now that we have the obvious bound
\begin{equation}\label{annulus-measure}
\omega_j\leq\eta j^{n-1}T^n,
\end{equation}
where $\eta$ is the measure of the unit sphere in $\R^n$.
Therefore we can derive for the right hand side of (\ref{basic-inequality}) the upper bound
\begin{equation}\label{right-inductive}
\sum_{j=1}^k\epsilon^j\omega_{k+1-j}+\omega_{k+1}\leq\eta T^n k^{n-1}\sum_{j=1}^k\epsilon^j+\omega_{k+1}
\leq\eta T^n k^{n-1}\frac{\epsilon}{1-\epsilon}+\omega_{k+1}
\end{equation}
From this and (\ref{left-inductive-1}) we get
\begin{equation}\label{leftright-inductive}
\frac{C_0}{2^n}(\frac{c^*}{n})^\frac{n-1}{n}(k+1)^{n-1}\leq\eta T^n\frac{\epsilon}{1-\epsilon}k^{n-1}+\omega_{k+1}.
\end{equation}
Since $\epsilon=e^{-c_2T}$ we can choose $T>0$ so large that
\[\eta T^n\frac{\epsilon}{1-\epsilon}\leq\frac{C_0}{2^{n+1}}(\frac{c^*}{n})^\frac{n-1}{n}.\]
Then from (\ref{leftright-inductive}) we obtain
\begin{equation}\label{near-the-end}
\frac{C_0}{2^{n+1}}(\frac{c^*}{n})^\frac{n-1}{n}(k+1)^{n-1}\leq\omega_{k+1}.
\end{equation}
Therefore to complete the induction it suffices to observe that we can choose $c^*$ so small that
\[c^*\leq\frac{C_0}{2^{n+1}}(\frac{c^*}{n})^\frac{n-1}{n}\;\Leftrightarrow\;
1\leq\frac{C_0}{2^{n+1}n^\frac{n-1}{n}(c^*)^\frac{1}{n}}.\]
Let $[R/T]$ the integer part of $R/T$ and observe that
\[\frac{[R/T]}{R/T}\geq\frac{1}{2},\;\text{ for }\;R\geq T.\]
From (\ref{claim}) and (\ref{left-inductive}) we have
\[\mathcal{L}^n(B_R\cap\{q^u>\lambda^*\})\geq\sum_{k=1}^{[R/T]}\omega_k\geq\frac{c^*}{nT^n}(\frac{[R/T]}{R/T})^nR^n\geq
\frac{c^*}{n2^nT^n}R^n,\;\text{ for }\;R\geq T.\]
\end{proof}
\noindent  Claim 6 concludes the case $\alpha=2$ and completes the proof of Theorem A for small $\lambda>0$. As in the case $\alpha<2$ the restriction on the smallness of $\lambda$ is removed via (\ref{energy-upper-bound-1}).

\section{Pointwise Estimates-Liouville type results}

 Theorem A implies the following basic estimate (cfr. Theorem 1.2 in Fusco \cite{fu})
\begin{theorem}\label{teo-f}
Assume that $W$ satisfies {\rm({HA})}  and assume that $u:D\to\R^m$ is minimal in the sense of {\rm({HB})}, $D\subset\R^n$ open. Let $\mathcal{Z}:=\{ W=0\}\setminus\{a\}$ and assume
\begin{equation}\label{mathzeta}
\mathcal{Z}=\emptyset\;\;\text{ or }\;\;d_0=\inf_{x\in D}d(u(x),\mathcal{Z})>0,\;d\;\text{ the Euclidean distance }.
\end{equation}
Then, given $\lambda>0$, there is
$R(\lambda)$ such that
\begin{eqnarray}\label{condition}
B_{R(\lambda)}(x_0)\subset D,\;\;\Rightarrow\;\;\vert u(x_0)-a\vert< \lambda.
\end{eqnarray}
 $R(\lambda)$ depends only on $W$  and on the bound $M$ in {\rm({HB})} if $\mathcal{Z}=\emptyset$ and also on $d_0$ otherwise.
\end{theorem}
\begin{proof}
Let $R_{x_0}=\max\{R:B_R(x_0)\subset D\}$ and assume $R(x_0)>1$. Then,
from (\ref{grad-bound}), we have that the inequality
\[\vert u(x_0)-a\vert\geq\lambda\]
 implies
\[\mathcal{L}^n(B_1(x_0)\cap\{\vert u(x)-a\vert\geq\frac{\lambda}{2}\}\geq\mu_0>0\]
and therefore Theorem A yields
\begin{equation}
\label{inequality-}\mathcal{L}^n(B_R(x_0)\cap\{\vert u(x)-a\vert\geq\frac{\lambda}{2}\}\geq \tilde{C} R^n,\;\text{ for }\;1<R<R_{x_0}
\end{equation}
 and a constant $\tilde{C}=\tilde{C}(\lambda,M)>0$ independent of $x_0$.
Observe that the assumption (\ref{mathzeta})
implies via (\ref{inequality-})
\begin{equation}\label{wbar-below}\bar{w}\tilde{C} R^n\leq\int_{B_R(x_0)}W(u)dx\leq J_{B_R(x_0)},\;\text{ for }\;R\leq R_{x_0}\end{equation}
where we have set
\[\bar{w}=\min\{W(z): \vert z-a\vert\geq\frac{\lambda}{2},\;d(z,\mathcal{Z})\geq d_0,\;\vert z-a\vert\leq M\}>0.\]
The inequality (\ref{wbar-below}) and the upper bound (\ref{energy-upper-bound-1}) in Lemma \ref{energy-upper-bound} are compatible only if $R\leq\frac{C}{\bar{w}\tilde{C}}$ where $C$ is the constant in Lemma \ref{energy-upper-bound}. Therefore if
\[R_{x_0}\geq 2\frac{C}{\bar{w}\tilde{C}}\]
we necessarily have
\[\vert u(x_0)-a\vert<\lambda.\]
This concludes the proof with $R(\lambda)=2\frac{C}{\bar{w}\tilde{C}}$.
\end{proof}

Theorem \ref{teo-f} allows to extend to potentials that satisfy ({HA}) and in particular to singular potentials ($\alpha\in(0,1]$) the following  {\it Liouville} type result established in \cite{fu}.
\begin{theorem}\label{rigidity}
Let $W$ and $u$ be as in Theorem \ref{teo-f} and assume $D=\R^n$. Then
\[u\equiv a.\]
\end{theorem}
\begin{proof} $D=\R^n$ trivially implies that, given $x_0\in\R^n$ and $\lambda>0$, $B_{R(\lambda)}(x_0)\subset D$. Then Theorem \ref{teo-f} yields
\[\vert u(x_0)-a\vert<\lambda,\;\text{ for }\;\lambda>0\;x_0\in\R^n.\]
The proof is complete.
\end{proof}
The following exponential estimate ( see \cite{fu} Theorem 1.3)  can be considered a consequence of the density estimate in Theorem A.
\begin{theorem}\label{1exp}
Let $u:D\rightarrow\R^m$ and $W$ as in Theorem \ref{teo-f}. Assume $\alpha=2$ in {\rm(HA)} and $D\neq\R^n$ with $\sup_{x_0\in D}R_{x_0}=+\infty$. Then
\[\vert u(x)-a\vert\leq K e^{-k d(x,\partial D)},\;\text{ for some }\;k,K>0.\]
\end{theorem}
\begin{proof}
First we note that it is sufficient to establish that, given a small number $\lambda>0$, there is $d_\lambda>0$ such that
\[d(x,\partial D)\geq d_\lambda\;\;\Rightarrow\;\;\vert u(x)-a\vert\leq\lambda\]
since then linear theory renders the result. From Theorem \ref{teo-f} it follows that we can take $d_\lambda=R(\lambda)$.
  The proof is complete.
\end{proof}

\section{On the Linking with the Minimal Surface Problem}
We will consider partitions with Dirichlet conditions for simplicity. The volume constraint case is more involved but similar. Assume that $W$ is as in Theorem \ref{main-1} and that therefore
\[0=W(a_1)=W(a_2)<W(u),\;u\not\in\{a_1,a_2\}\]
for $a_1\neq a_2\in\R^m$.
Let $\{u_{\epsilon_k}\}$ be a sequence of global minimizers of $J_D^\epsilon(u)=\int_D(\frac{\epsilon^2}{2}\vert\nabla u\vert^2+W(u))dy$ subject to the Dirichlet condition $u_{\epsilon_k}=g$ on $\partial D$, $g:\partial D\rightarrow\{a_1,a_2\}$.

We assume $D\subset\R^n$ open bounded with $C^1$ boundary and consider a partition of the boundary $B_j=g^{-1}(\{a_j\}),\;j=1,2$ with $\mathcal{H}^{n-1}(\partial D\setminus(B_1\cup B_2))=0$. We also assume that $\|u_{\epsilon_k}\|_{L^\infty(D;\R^m)}<M$ uniformly. Then by the methods in Baldo \cite{b} $\{u_{\epsilon_k}\}$ is relatively compact in $L^1(D;\R^m)$ and along a subsequence $\epsilon_k\rightarrow 0^+$ $u_{\epsilon_k}\stackrel{L^1}{\rightarrow}u_0=a_1\chi_{D_1}+a_2\chi_{D_2}$ where $D_1,D_2$ is a partition of $D$ with $\partial D_j\cap\partial D=B_j,\;j=1,2$. Moreover the interface $\partial D_1\cap\partial D_2$ minimizes $\mathcal{H}^{n-1}(\partial A_1\cap\partial A_2)$ among all partitions of $D$ with Dirichlet conditions $B$. For two-phase partitions, if $n\leq 7$, the interface $\partial D_1\cap\partial D_2$ is locally a real analytic classical minimal surface (see \cite{gmt1}).

We write $u_\epsilon$ in polar form (cfr. (\ref{polar-form})), $u_\epsilon=a_1+\rho_\epsilon\nu_\epsilon$ with
\[\rho_\epsilon(y)=\vert u_\epsilon-a_1\vert,\;\;\nu_\epsilon(y)=\frac{u_\epsilon-a_1}{\vert u_\epsilon-a_1\vert},\;\nu_\epsilon(y)=0\;\text{ if }\;\rho_\epsilon(y)=0.\]
Then, from $u_{\epsilon_k}\stackrel{L^1}{\rightarrow}u_0$, we obtain that
\begin{eqnarray}\label{ro-limit}
\rho_{\epsilon_k}\rightarrow\rho_0=\left\{\begin{array}{l}
0\;\text{ in }\;D_1,\\
\vert a_1-a_2\vert\;\text{ in }\;D_2.
\end{array}\right.
\end{eqnarray}
\begin{proposition}\label{uni-conv}
The level set $\mathcal{S}_\epsilon=\{y\in D:\vert u_\epsilon-a_j\vert=\gamma,\;j=1,2\}$, $\gamma\in(0,\vert a_1-a_2\vert)$ converges locally uniformly to $\partial D_1\cap\partial D_2$ as $\epsilon\rightarrow 0^+$.
\end{proposition}
\begin{proof}
(Blow-up, cfr. Theorem 2 in \cite{cc}) Suppose that the convergence is not uniform over a compact set $\mathcal{K}\subset\subset D$. Then there are sequences $\epsilon_k\rightarrow 0^+$, $y_k\in\mathcal{S}_{\epsilon_k}\cap\mathcal{K}$, $k=1,\ldots$ and $r>0$ such that  $d(y_k,\partial D_1\cap\partial D_2)\geq r$. We can assume that all the points $y_k$ are in one of the sets $D_j,\;j=1,2$. For definiteness we suppose $y_k\in D_1,\;k=1,\ldots$. Actually $\mathcal{K}\subset\subset D$ implies that we can assume $B_r(y_k)\subset D_1,\;k=1,\ldots$. Set
\[x=\frac{y-y_k}{\epsilon_k},\;\;\;\;
v_k=u_{\epsilon_k}(\epsilon_kx+y_k),\;\;\;\;\varrho_k(x)=\rho_{\epsilon_k}(\epsilon_kx+y_k).\]
Since $u_{\epsilon_k}$ is a  minimizer we have $\Delta v_k-W(v_k)=0$, $\varrho_k(0)=\gamma$ and $\vert v_k-a\vert<M$.  Thus we also have the gradient bound $\vert\nabla v_k\vert<M$ which implies
\[\varrho_k(x)>\frac{\gamma}{2}\;\text{ for }\;\vert x\vert<\delta\]
for some $\delta>0$ independent of $k=1,\ldots$.
Now we observe that $v_k,\;k=1,\ldots$ is a  minimizer of $J_{D_k}(v)=\int_{D_k}(\frac{1}{2}\vert\nabla v\vert^2+W(v))dx$, $D_k=\{x=y-y_k/\epsilon_k,y\in D\}$. Thus we can apply Theorem A that yields the density estimate
\begin{equation}\label{dens-est-x}
\mathcal{L}^n(\{\vert x\vert<R\}\cap\{\varrho_k(x)>\frac{\gamma}{2}\})\geq C R^n,\;R\geq\delta
\end{equation}
that holds uniformly over the family $\{v_k\}$. This estimate is equivalent to
\[\mathcal{L}^n(B_{\epsilon_kR}(y_k)\cap\{\rho_{\epsilon_k}(y)>\frac{\gamma}{2}\})\geq C (\epsilon_k R)^n,\;R\geq\delta.\]
In particular, for $R=r/\epsilon_k$, we get
\begin{equation}
\label{dens-est-y}
\mathcal{L}^n(B_r(y_k)\cap\{\rho_{\epsilon_k}(y)>\frac{\gamma}{2}\})\geq C r^n.
\end{equation}
Since $B_r(y_k)\subset D_1$ and $\rho_0=0$ a.e. on $D_1$ (\ref{dens-est-y}) implies
\[\int_D\vert\rho_{\epsilon_k}-\rho_0\vert dy\geq\int_{B_r(y_k)}\rho_{\epsilon_k}dy\geq\frac{\gamma}{2}C r^n\]
which contradicts (\ref{ro-limit}). The proof is complete.
\end{proof}
\section{A Lower Bound for the Energy}
In this section we adopt the following hypothesis
\begin{description}
\item[({HC})] There exists $N\geq 2$ and $N$ distinct points $a_1,\ldots,a_N\in\R^m$ such that
\[0=W(a_j)<W(u),\;j=1,\ldots,N,\;u\in\R^m\setminus\{a_1,\ldots,a_N\}.\]
Moreover $W:\R^m\rightarrow\R$ is as in ({HA}) for $a=a_j,\;j=1,\ldots,N$.
\end{description}
From the monotonicity formula (see (1.4) in \cite{a}), which holds for general Lipschitz $W\geq 0$, it follows that {\it any} solution to $\Delta u-W_u(u)=0$ satisfies the lower bound
\begin{equation}\label{weak-lb}
J_{B_R(x_0)}(u)\geq R^{n-2}J_{B_1(x_0)}(u),\;R\geq 1.
\end{equation}
If $W(u)=(1-\vert u\vert^2)^2$ and, more generally, if the set of the zeros of $W$ is not totally disconnected, the lower bound above is sharp (see (2.4) in Farina \cite{fa}). On the other hand for the class of phase transition potentials defined in ({HC}) above, under the hypothesis of minimality we have
\begin{proposition}\label{lb}
Let $u:\R^n\rightarrow\R^m$ be nonconstant and minimal in the sense of {\rm({HB})}, and pointwise bounded uniformly over $\R^n$ {\rm(cfr.  (\ref{grad-bound}))}. Then we have
\begin{equation}\label{lb-1}
\int_{B_R(x_0)}\Big(\frac{1}{2}\vert\nabla u\vert^2+W(u)\Big)dx\geq C R^{n-1},\;R\geq R(x_0).
\end{equation}
with $C>0$ independent of $x_0$.
\end{proposition}
\begin{proof}
Since $u$ is continuous and nonconstant there are $\gamma>0$ and $\xi\in\R^n$ such that $\vert u(\xi)-a_j\vert>\gamma,\;j=1,\ldots,N$. Thus $\mathcal{L}^n(B_1(\xi)\cap\{\vert u(\xi)-a_j\vert>\gamma/2\})\geq\mu_0,\;j=1,\ldots,N$ for some $\mu_0>0$ and so by Theorem A
\begin{equation}\label{muj-rn}
\mathcal{L}^n(B_R(\xi)\cap\{\vert u(\xi)-a_j\vert>\frac{\gamma}{2}\})\geq C R^n,\;R\geq 1,\;j=1,\ldots,N.
\end{equation}
This and the same argument as in the proof of Theorem \ref{main-1} imply
\[\sum_{i\neq j}\mathcal{L}^n(B_R(\xi)\cap\{\vert u(\xi)-a_i\vert<\frac{\gamma}{2}\})\geq C R^n,\;R\geq 1,\;j=1,\ldots,N.\]
It follows that, for each $R\geq 1$, at least for two distinct $a_-,a_+\in\{a_1,\ldots,a_N\}$ we have
\begin{equation}\label{two-big-sets}
\mathcal{L}^n(B_R(\xi)\cap\{\vert u-a_\pm\vert<\frac{\gamma}{2}\})\geq C R^n.
\end{equation}
Define $\varphi:\R^n\rightarrow\R$ by setting $\varphi(x)=d(a_-,u(x))$ with $d(z_1,z_2)$ \footnote{ $d(z_1,z_2)$ is he geodesic distance (\cite{b} pag. 71)} given by.
\[d(z_1,z_2)=\inf\{\int_0^1\sqrt{2W(\zeta(s))}\vert\zeta^\prime(s)\vert ds,\;\zeta\in C^1([0,1];\R^m),\;\zeta(0)=z_1,\,\zeta(1)=z_2\}.\]
By (\ref{grad-bound}) we have $d(a_\pm,z)\leq C\vert z-a_\pm\vert$ with $C=\max_i\max_{\vert z-a_i\vert\leq M}\sqrt{2W(z)}$. It follows that, provided $\gamma\in(0,d(a_-,a_+)/C)$,
\begin{equation}\label{are-inside}
\begin{split}
&\{\vert u-a_-\vert<\frac{\gamma}{2}\}\subset\{d(a_-,u)<t\},\\
&\{\vert u-a_+\vert<\frac{\gamma}{2}\}\subset\{d(a_-,u)>t\},
\end{split}\text{ for }\;t\in(C\frac{\gamma}{2},d(a_-,a_+)-C\frac{\gamma}{2}).
\end{equation}
This and the relative isoperimetric inequality ( see \cite{eg} pag. 190)
\[\min\{\mathcal{L}^n(B_R(\xi)\cap\{\varphi<t\}),\mathcal{L}^n(B_R(\xi)\setminus\{\varphi<t\})\}^\frac{n-1}{n}\leq C\mathcal{H}^{n-1}(B_R(\xi)\cap\varphi^{-1}(t))\]
imply via (\ref{two-big-sets}) the estimate
\begin{equation}\label{relative}
C R^{n-1}\leq\mathcal{H}^{n-1}(B_R(\xi)\cap\varphi^{-1}(t))\;\text{ for }\;t\in(\alpha,\beta)
\end{equation}
where $\alpha=C\frac{\gamma}{2}$ and $\beta=d(a_-,a_+)-C\frac{\gamma}{2}$.
From Proposition 2.1 in \cite{b} and (\ref{grad-bound}) we have that $\varphi$ is lipschitz and
\begin{equation}\label{2inequality}
\int_A\vert D\varphi\vert dx\leq\int_A\sqrt{2W(u)}\vert D u\vert dx
\end{equation}
for all bounded smooth open subsets $A\subset\R^n$. Therefore by the coarea formula, the estimate (\ref{relative}), and Young's inequality we obtain
\[\begin{split}& C R^{n-1}\leq \int_\alpha^\beta\mathcal{H}^{n-1}(B_R(\xi)\cap\varphi^{-1}(t))dt=\int_{B_R(\xi)\cap\{\alpha<\varphi(x)<\beta\}}\vert D\varphi\vert dx\\&\leq\int_{B_R(\xi)}\sqrt{2W(u)}\vert D u\vert dx\leq J_{B_R(\xi)}(u)
\end{split}\]
that concludes the proof.
\end{proof}
We give another proof of Proposition \ref{lb} via linking with the sharp interface problem in \cite{b}.
\begin{proof}(Blow-down) Let $\xi\in\R^n$ as before and set $x-\xi=\frac{y}{\epsilon}$,  $u_\epsilon(y)=u(\xi+\frac{y}{\epsilon})$. Then (\ref{energy-upper-bound-1}) implies
\begin{equation}\label{u-b-eps}
\int_{\vert y\vert<r}\Big(\frac{\epsilon}{2}\vert\nabla u_\epsilon\vert^2+\frac{1}{\epsilon}W(u_\epsilon)\Big) dy\leq C r^{n-1},\;\text{ for }\;\epsilon\in(0,1),
\end{equation}
where $r=\epsilon R$ is {\underline {fixed}} once and for all. From (\ref{grad-bound}) and (\ref{u-b-eps}) it follows (see pages 73, 82 in \cite{b}) that $\|u_\epsilon\|_{{\mathrm BV}(B_r(0);\R^m)}<C$ and so along a subsequence
\[u_{\epsilon_k}\stackrel{L^1}{\,\rightarrow\,}u_0\;\text{ in }\;B_r(0)\;\text{ and }\;u_0(y)\in\{a_1,\ldots,a_N\}\;\text{ a.e. }\]
Moreover $A_j=\{u_0(y)=a_j\},\;j=1,\ldots,N$ are sets of finite perimeter in $B_r(0)$. From (\ref{muj-rn}) we have **
\[\mathcal{L}^n(B_r(0)\cap\{\vert u_{\epsilon_k}-a_j\vert>\frac{\gamma}{2}\})\geq C(\epsilon_k R)^n=C r^n,\;\text{ for }\;j=1,\ldots,N.\]
and by passing to the limit for $k\rightarrow\infty$ we obtain
\[\mathcal{L}^n(B_r(0)\cap\{\vert u_0-a_j\vert>\frac{\gamma}{2}\})\geq C r^n,\;\text{ for }\;j=1,\ldots,N.\]
Hence
\[\sum_{i\neq j}\mathcal{L}^n(B_r(0)\cap A_i)\geq C r^n,\;\text{ for }\;j=1,\ldots,N.\]
From this it follows that at least for two distinct values $a_h\neq a_l$ the sets $A_h, A_l$ have full measure:
\[\mathcal{L}^n(B_r(0)\cap A_h)\geq\frac{C}{N-1}r^n,\quad \mathcal{L}^n(B_r(0)\cap A_l)\geq\frac{C}{N-1}r^n.\]
Then the relative isoperimetric inequality implies
\begin{equation}\label{2rel}
\mathcal{H}^{n-1}(\partial A_h)\geq C r^{n-1},\quad\mathcal{H}^{n-1}(\partial A_l)\geq C r^{n-1}
\end{equation}
where $\partial A_j$ is the the relative boundary of $A_j$ in $B_r(0)$ and $C>0$ a constant. Finally by lower semicontinuity (see pag.76 in \cite{b}) and (\ref{2rel}) we have
\begin{equation}\label{lower-sem}
\begin{split}
\liminf_{k\rightarrow\infty}\int_{\vert y\vert<r}\Big(\frac{\epsilon_k}{2}\vert\nabla u_{\epsilon_k}\vert^2+\frac{1}{\epsilon_k}W(u_{\epsilon_k})\Big) dy &\geq\sum_{i,j=1,\,i\neq j}^Nd(a_i,a_j)\mathcal{H}^{n-1}(\partial A_i\cap\partial A_j)\\&\geq C\mathcal{H}^{n-1}(\partial A_h)\geq C r^{n-1}.
\end{split}
\end{equation}
Since the right hand side of (\ref{lower-sem}) is independent of the particular subsequence $\{\epsilon_k\}$ considered, we conclude that there is $\epsilon_0>0$ such that
\[\int_{\vert y\vert<r}\Big(\frac{\epsilon}{2}\vert\nabla u_\epsilon\vert^2+\frac{1}{\epsilon}W(u_\epsilon\Big) dy\geq C(\frac{r}{2})^{n-1},\;\text{ for }\;\epsilon\in(0,\epsilon_0)\]
and in the original variables
\begin{equation}\label{bound-xi}
\int_{B_R(\xi)}\Big(\frac{1}{2}\vert\nabla u\vert^2+W(u)\Big) dx\geq C R^{n-1},\;\text{ for }\;R\geq R_0.
\end{equation}
To conclude the proof we show that given $x_0\in\R^n$ there is $R(x_0)$ such that
\[J_{B_R(x_0)}(u)\geq\frac{C}{2}R^{n-1},\;\text{ for }\;R\geq R(x_0)\]
where $C>0$ is the constant in (\ref{bound-xi}). Indeed from (\ref{bound-xi}), for $R\geq R_0+\vert x_0-\xi\vert$, we have, with $d=\vert x_0-\xi\vert$,
\[J_{B_R(x_0)}(u)\geq J_{B_{R-d}(\xi)}(u)\geq C (R-d)^{n-1}\geq\frac{C}{2}R^{n-1}\;\text{ for }\;R\geq R(x_0)=
\frac{2^\frac{1}{n-1}}{2^\frac{1}{n-1}-1}d.\]
This completes the proof.
\end{proof}

\newpage

\centerline{\underline{\bf PART II}}
\section{Theorem B}
\subsection{Hypotheses and statement}
In this subsection we consider
\begin{equation}\label{2system}
\Delta u-W_u(u)=0,\quad u:\R^n\rightarrow\R^m,
\end{equation}
in the class of {\it symmetric} solutions
\[u(\hat{x})=\hat{u}(x)\]
where for $z\in\R^d,\;d\geq 1$ we denote by $\hat{z}$ the symmetric of $z$ in the plane $\{z_1=0\}$ that is
\[\hat{z}=(-z_1,z_2,\ldots,z_d).\]
We assume that $W:\R^m\rightarrow\R$ is a $C^3$ potential that satisfies
\begin{description}
\item[({ha})] $W$ is symmetric: $W(\hat{u})=W(u),\;\text{ for }\;u\in\R^m$ and
\begin{equation}\label{ha-w}
0=W(a_+)<W(u),\;\text{ for }\;u\in\overline{\R_+^m}\setminus\{a_+\}
\end{equation}
for a unique $a_+\in\R_+^m=\{u\in\R^m:u_1>0\}$.
\begin{equation}
 W_{uu}(a_+)\nu\cdot\nu\geq C_0>0,\;\text{ for }\;\vert\nu\vert=1,\hskip2cm
\label{wu-lower-bound-2}
\end{equation}
where $W_{uu}(a_+)$ is the Hessian matrix of $W$ at $a$.
\item[({hb})] There exists $e:\R\rightarrow\R^m$ ({\it connection}) satisfying
\begin{eqnarray}
\left\{\begin{array}{l}
 e_{ss}=W_u(e),\quad s\in\R \\
 e(-s)=\hat{e}(s),\;s\in\R,\\
 \lim_{s\rightarrow\pm\infty}e(s)=a_\pm,
\end{array}\right.
\end{eqnarray}
which moreover is a {\it global minimizer} of the {\it Action } functional
\[ A(v)=\int_\R(\frac{1}{2}\vert v_s\vert^2+W(v))ds\]
in the class of $v\in W_{\rm loc}^{1,2}(\R;\R^m)\cap L^\infty(\R;\R^m)$
which are symmetric and satisfy $\lim_{s\rightarrow\pm\infty}v(s)=a_\pm$.

The connection $e$ is {\it hyperbolic} in the class of {\it symmetric} perturbations in the sense that
 the operator $T$ defined by
\begin{eqnarray}
D(T)=W_S^{2,2}(\R,\R^m),\quad\quad Tv=-v^{\prime\prime}+W_{uu}(e)v,
\end{eqnarray}
where $W_S^{2,2}(\R,\R^m)\subset W^{2,2}(\R,\R^m)$ is the subspace of symmetric maps, satisfies
\begin{equation}\label{t-spectral-bound}
\langle Tv,v\rangle\geq\eta\|v\|^2,\;v\in W_S^{1,2}(\R,\R^m).
\end{equation}
for some $\eta>0$. Here $\langle,\rangle$ is the  inner product in $L^2(\R;\R^m)$ and $\|\;\|$ the associated norm and $W_{uu}$ is the Hessian matrix of $W$.
\item[({hc})]
 $u\in W_{\rm loc}^{1,2}(\R^n;\R^m)\cap L^\infty(\R^n;\R^m)$, is  {\it minimal} in the class of symmetric maps in the sense that
 \begin{equation*}
 J_\Omega(u)\leq J_\Omega(u+v),\;\text{ for each symmetric }\;v\in W_0^{1,2}(\Omega;\R^m)
 \end{equation*}
 and for every open symmetric bounded lipschitz set $\Omega\subset\R^n$.
 Moreover $u$ satisfies the estimates
 \begin{equation}
 \vert u-a\vert+\vert\nabla u\vert\leq Ke^{-k x_1},\;\text{ on }\;\overline{\R_+^n}
 \label{grad-bound-1}
 \end{equation}
 for some $k,K>0$.
\end{description}
Since we have
 \begin{equation}
 \vert e-a\vert+\vert u_{x_1}\vert\leq Ke^{-k x_1},\;\text{ on }\;x_1\geq 0,
 \label{grad-bound-e}
 \end{equation}
 it follows, via (\ref{grad-bound-1}), that
 \begin{equation}\label{2qnu-bound}
 \|u(\cdot,x_2,\ldots,x_n)-e(\cdot\,)\|_{W^{1,2}(\R;\R^m)}\leq M_1,\;\text{ for }\;(x_2,\ldots,x_n)\in\R^{n-1}
 \end{equation}
 for some constant $M_1>0$.

 \noindent We denote $\mathrm{E}^{\mathrm{xp}}\subset W_{S,\mathrm{loc}}^{1,2}(\R;\R^m)$ the exponential class of symmetric maps which, as $e$, satisfy (\ref{grad-bound-e}) with $k, K>0$ fixed constants.
\vskip.3cm
\noindent\underline{{\it Notes}}
\begin{enumerate}
\item Under hypotheses ({ha}) by Theorems 3.6, 3.7 in \cite{af3} there is a connection $e$ symmetric and global minimizer of $A$.

\item In the proof of Theorem B we utilize minimality only in symmetric cylinders.
\end{enumerate}
\noindent {\underline{\it Notation}}

As before by $\cdot$ we denote the Euclidean inner product in $\R^d$ $d\geq 2$. We write the typical $x=(x_1,\ldots,x_n)\in\R^n$ in the form $x=(s,y)$ with $s=x_1\in\R$ and $y=(x_2,\ldots,x_n)\in\R^{n-1}$. For $r>0$ and $y_0\in\R^{n-1}$ we set $\mathcal{B}_r(y_0)=\{y\in\R^{n-1}:\vert y-y_0\vert<r\}$. By $C_r(y_0)\subset\R^n$ we denote the cylinder $\R\times \mathcal{B}_r(y_0)$.

\begin{theoremB}\label{main-3}
Under hypothesis {\rm({ha})}, {\rm({hb})} and {\rm({hc})} above, there exists $\lambda^*>0$ small, independent of $u$, such that for any $\mu_0>0$ and any $0<\lambda<\lambda^*$ the condition
\begin{equation}
\mathcal{L}^{n-1}(\mathcal{B}_1(y_0)\cap\{y:\| u(\cdot,y)-e(\cdot\,)\|\geq\lambda\})\geq\mu_0\hskip1.5cm
\label{2mu0-cond}
\end{equation}
implies the estimate
\begin{equation}
\mathcal{L}^{n-1}(\mathcal{B}_R(y_0)\cap\{y:\| u(\cdot,y)-e(\cdot\,)\|\geq\lambda\})\geq C R^{n-1},\;\text{ for }\;R\geq 1
\label{2mu-cond}
\end{equation}
where $C=C(\mu_0,\lambda,K)$, is independent of $y_0$ and independent of $u$.
\end{theoremB}
Theorem B has the following important consequence
\vskip.3cm
\begin{theorem}\label{main-4}
Assume that $W$ satisfies {\rm({ha})} and that $u:\R^n\rightarrow\R^m$ is minimal in the sense of {\rm({hc})}. Assume that there are exactly two global minimizers $e_+\neq e_-$ of the action $A$ in the symmetric class with the properties of $e$ in {\rm({hb})} above.
Then the condition
\begin{equation}
\mathcal{L}^{n-1}(\mathcal{B}_1(y_0)\cap\{y:\| u(\cdot,y)-e_+(\cdot\,)\|\leq\theta\})\geq \mu_0>0\hskip1.5cm
\label{2mu0-cond-1}
\end{equation}
$\theta\in(0,\|e_+-e_-\|)$, arbitrary otherwise, implies the estimate
\begin{equation}
\mathcal{L}^{n-1}(\mathcal{B}_R(y_0)\cap\{y:\| u(\cdot,y)-e_+(\cdot\,)\|\leq\theta\})\geq C R^{n-1},\;\text{ for }\;R\geq 1,
\label{2mu-cond-1}
\end{equation}
where $C=C(\mu_0,\lambda,K)$, is independent of $y_0$ and independent of $u$. An analogous statement applies to $e_-$.
\end{theorem}
\subsection{The Proof of Theorem B}
\vskip.3cm
1.\underline{The Polar Form and the Effective Potential}

We will utilize the  {\it polar form  with respect to $e$} of a vector map $u\in W_{\mathrm{loc}}^{1,2}(\R^n;\R^m)\cap L^\infty(\R^n;\R^m)$. We write
\[u(s,y)=e(s)+q^u(y)\nu^u(s,y),\;(s,y)\in\R^n\]
where
\[q^u(y)=\|u(\cdot,y)-e(\cdot\,)\|\]
and
\begin{eqnarray}\label{2polar}
\nu^u(\cdot,y)=\left\{\begin{array}{l}
\frac{u(\cdot,y)-e(\cdot\,)}{\|u(\cdot,y)-e(\cdot\,)\|}\;\text{ if }\;q^u(y)\neq 0,\\\\
0,\;\;\text{ otherwise }.
\end{array}\right.
\end{eqnarray}
We have
\begin{equation}\label{2grad}
\frac{\partial u}{\partial y_i}=\frac{\partial q^u}{\partial y_i}\nu^u+q^u\frac{\partial \nu^u}{\partial y_i}
\end{equation}
and therefore observing that
\begin{equation}\label{2nu-propoerties}
\|\nu^u(\cdot,y)\|=1,\;\;\langle\nu^u(\cdot,y),\frac{\partial \nu^u}{\partial y_i}(\cdot,y)\rangle=0,\;i=1,\ldots,n-1,
\end{equation}
we obtain the following {\it polar representation} of the energy of $u$
\begin{equation}\label{2energy-polar0}
\begin{split}
&\int_{C_r(y_0)}\Big(\frac{1}{2}\vert\nabla u\vert^2+W(u)\Big) dx\\&=\int_{\mathcal{B}_r(y_0)}\Big(\frac{1}{2}\Big(\vert\nabla q^u\vert^2+(q^u)^2\sum_{i=1}^{n-1}\|\frac{\partial \nu^u}{\partial y_i}\|^2\Big)+\mathcal{W}(u)+A(e)\Big) dy
\end{split}
\end{equation}
where $\mathcal{W}:e+W^{1,2}(\R;\R^m)\rightarrow\R$ the {\it Effective Potential} is defined by
\begin{equation}\label{effective-pot}
\begin{split}
&\mathcal{W}(v)=A(v)-A(e)\\
&=\int_\R\Big(\frac{1}{2}(\|v_s\|^2-\|e_s\|^2)+W(v)-W(e)\Big) ds,\;\text{ for } v-e\in W_S^{1,2}(\R;\R^m).
\end{split}
\end{equation}
As it is standard in variational arguments, adding a constant to the integrand in (\ref{2energy-polar0}) does not affect what follows. Therefore we disregard the constant $A(e)$ in (\ref{2energy-polar0}) and define the modified energy $J_{C_r(y_0)}(u)$ by setting
\begin{equation}\label{2energy-polar}
J_{C_r(y_0)}(u)=\int_{\mathcal{B}_r(y_0)}\Big(\frac{1}{2}\Big(\vert\nabla q^u\vert^2+(q^u)^2\sum_{i=1}^{n-1}\|\frac{\partial \nu^u}{\partial y_i}\|^2\Big)+\mathcal{W}(u)\Big) dy
\end{equation}
where we have slightly abused the notation in (\ref{energy-def}). Note that
\begin{equation}\label{je0}
J_{C_r(y_0)}(e)=0.
\end{equation}

\newpage
\begin{lemma}\label{2w-properties}
We have
\begin{description}
\item[(i)]
$\mathcal{W}\geq 0$.
\item[(ii)]
Let $\SF=W^{1,2}(\R;\R^m)\cap\{\|v\|=1\}$. Assume that $v(s)=e(s)+q\nu,\;q\in\R,\,\nu\in\SF$ satisfies (\ref{2qnu-bound}). Then there are constants $c_0>0$ and $\bar{q}>0$ such that
\begin{equation}\label{2w-derivative}
D_{qq}\mathcal{W}(e+q\nu)\geq c_0,\;\text{ for }\;q\in[0,\bar{q}],\;\nu\in\SF.
\end{equation}
\end{description}
\end{lemma}
\begin{proof}
(i) follows from ({hb}). To prove (ii) we begin
by differentiating twice $\mathcal{W}(e+q\nu)$ with respect to $q$. We obtain
\begin{eqnarray}\label{2second-d}
D_{qq}\mathcal{W}(e+q\nu)&=&\|\nu_s\|^2+\int_\R W_{uu}(e+q\nu)\nu\cdot\nu ds\\\nonumber
&=&
D_{qq}\mathcal{W}(e+q\nu)|_{q=0}+\int_\R (W_{uu}(e+q\nu)-W_{uu}(e))\nu\cdot\nu ds.
\end{eqnarray}
From the interpolation inequality:
\begin{equation}
\begin{split}
\|f\|_{L^\infty(\R;\R^m)}\leq &\sqrt{2}\|f\|^{\frac{1}{2}}\|f_s\|^{\frac{1}{2}},\\
\leq &\sqrt{2}\|f\|_{W^{1,2}(\R;\R^m)},
\end{split}f\in W^{1,2}(\R;\R^m),
\end{equation}
applied to $q\nu$
we obtain via the second inequality
\begin{equation}
\|q\nu\|_{L^\infty(\R;\R^m)}\leq \sqrt{2}M_1,
\end{equation}
with $M_1$ the constant in (\ref{2qnu-bound}),
and via the first
\begin{equation}
\|\nu\|_{L^\infty(\R;\R^m)}\leq \sqrt{2}M_1^{\frac{1}{2}}q^{-\frac{1}{2}},
\end{equation}
since $\|q\nu\|=q$ and $\|q\nu_s\|\leq M_1$.
Therefore we have
\begin{eqnarray}\label{w-uu-estimate}
\vert W_{u_iu_j}(e(s)+ q\nu(s))-W_{u_iu_j}(e(s))\vert\leq
\sqrt{2}M_1^{\frac{1}{2}}\overline{W}^{\prime\prime\prime}q^{\frac{1}{2}},
\end{eqnarray}
where $\overline{W}^{\prime\prime\prime}$ is defined by
\begin{eqnarray}
\overline{W}^{\prime\prime\prime}:=\max_{\left.\begin{array}{l}
1\leq i,j,k\leq m\\
s\in\R, \vert\tau\vert\leq 1
 \end{array}\right.}W_{u_iu_ju_k}(e(s)+ \tau\sqrt{2}M_1).
\end{eqnarray}
From (\ref{w-uu-estimate}) we get
\begin{eqnarray}\label{int-wuu-wuu-estimate}
\vert\int_\R(W_{uu}(e+q\nu)-W_{uu}(e))\nu\cdot\nu\vert ds\leq
C_1q^\frac{1}{2}\langle\nu,\nu\rangle=C_1q^\frac{1}{2},
\end{eqnarray}
where $C_1>0$ is a constant that depends on $M_1$.
We now observe that
\begin{eqnarray}\label{tl-equal-t}
D_{qq}\mathcal{W}(e+q\nu)|_{q=0}=\langle T\nu,\nu\rangle\geq\eta\|\nu\|^2=\eta
\end{eqnarray}
where we have also utilized (\ref{t-spectral-bound}).
Thus (\ref{tl-equal-t}) and (\ref{int-wuu-wuu-estimate}) in (\ref{2second-d}) yield
\begin{eqnarray}
D_{qq}\mathcal{W}(e+q\nu)\geq c_0:=\frac{\eta}{2},\;\;\text{ for }\,q\in[0,\bar{q}],
\end{eqnarray}
where $\bar{q}=\frac{1}{4}\frac{\eta^2}{C_1^2}$.
This concludes the proof of the lemma.
\end{proof}

In the following lemma we show that in the definition of minimality in ({hc}) we can extend the class of sets to include unbounded cylinders aligned to the $x_1$ axis.
\begin{lemma}\label{l-infty}
Let $u:\R^n\rightarrow\R^m$ be  minimal as in {\rm({hc})} above. Given a bounded open set $O\subset\R^{n-1}$, we have
\begin{equation}\label{min-infty}
J_{\R\times O}(u)=\min_{v\in u+W_{0 S}^{1,2}(\R\times O;\R^m)}J_{\R\times O}(v),
\end{equation}
where $W_{0 S}^{1,2}(\R\times O;\R^m)$ is the closure in $W_S^{1,2}(\R\times O;\R^m)$ of the smooth maps $v$ that satisfy $v=0$ on $\R\times\partial O$.
\end{lemma}
\begin{proof}
Assume there are $\eta>0$ and $v\in u+W_{0 S}^{1,2}(\R\times O;\R^m)$ such that
\begin{equation}\label{cont-assum}
J_{\R\times O}(u)-J_{\R\times O}(v)\geq\eta.
\end{equation}
For each $l>0$ define $\tilde{v}\in W_{0 S}^{1,2}(\R\times O;\R^m)$ by
\[
\tilde{v}=\left\{\begin{array}{l} v,\quad\text{ for }\;s\in[0,l],\;y\in O,\\
(1+l-s)v+(s-l)u, \;s\in[l,l+1],\;y\in O,\\
u,\quad\text{ for } \;s\in[l+1,+\infty),\;y\in O.
\end{array}\right.\]
The minimality of $u$ implies
\begin{equation}\label{before-limit}
0\geq J_{[-l-1,l+1]\times O}(u)-J_{[-l-1,l+1]\times O}(\tilde{v})=J_{[-l-1,l+1]\times O}(u)-J_{[-l,l]\times O}(v)+
\mathrm{O}(e^{-k l}),
\end{equation}
where we have also used the fact that both $u$ and $v$ belong to $W_S^{1,2}(\R\times O;\R^m)$ and satisfy (\ref{grad-bound-1}). Taking the limit for $l\rightarrow +\infty$ in (\ref{before-limit}) yields
\[0\geq J_{\R\times O}(u)-J_{\R\times O}(v),\]
in contradiction with (\ref{cont-assum}).
\end{proof}
For $q^h\in W^{1,2}(C_R;\R^m)\cap L^\infty(C_R;\R^m)$, $q^h\geq 0$, let the map $\sigma$  defined via $\sigma=e+q^\sigma\nu^u$, $q^\sigma=\min\{q^h,q^u\}$. We have $\sigma\in W^{1,2}(C_R;\R^m)\cap L^\infty(C_R;\R^m)$ \cite{afs}. The minimality of $u$ and the polar form (\ref{2energy-polar}) of the energy imply the inequality

 \begin{equation}\label{2young-1}
\begin{split}
&\frac{1}{2}\int_{\mathcal{B}_R}(\vert\nabla q^u\vert^2-\vert\nabla q^\sigma\vert^2) dy\\
&=J_{C_R}(u)-J_{C_R}(\sigma)
+\frac{1}{2}\int_{\mathcal{B}_R}\Big((q^\sigma)^2-(q^u)^2\Big)\sum_i^{n-1}\|\frac{\partial \nu^u}{\partial y_i}\|^2 dy+\int_{\mathcal{B}_R}(\mathcal{W}(\sigma)-\mathcal{W}(u))dy\\
&\leq\int_{\mathcal{B}_R}(\mathcal{W}(\sigma)-\mathcal{W}(u))dy.
\end{split}
\end{equation}
Indeed minimality and Lemma \ref{l-infty} imply $J_{C_R}(u)-J_{C_R}(\sigma)\leq 0$ and the second term is also nonpositive by $0\leq q^\sigma\leq q^u$.

\noindent 2.\underline{An Upper Bound for the Energy}

Next we establish the analogous of Lemma \ref{energy-upper-bound} that is
\begin{lemma}\label{2energy-upper-bound}
Assume that $W$ satisfies {\rm({ha})} and assume that $u$ is minimal as defined in {\rm({hc})} and $e$ a global minimizer of the Action as in {\rm({hb})} above (hyperbolicity is not required). Then there is a constant $C>0$ depending on $K$, independent of $u$ and independent of $y_0$ such that
\begin{equation}\label{2energy-upper-bound-1}
\begin{split}
& 0\leq\int_{C_r(y_0)}\Big(\frac{1}{2}\vert\nabla u\vert^2+W(u) -(\frac{1}{2}\vert e_s\vert^2+W(e))\Big) dx\\&=J_{C_R(y_0)}(u)\leq C R^{n-2},\;\text{ for }\;R>0.
 \end{split}
 \end{equation}
\end{lemma}
\begin{proof}
Let
\begin{eqnarray}\label{comp-map}
\hskip1cm v(\cdot,y)=\left\{\begin{array}{l}
e(\cdot\,),\quad\;\text{ for }\;y\in \mathcal{B}_{R-1}(y_0),\\\\
e(\cdot\,)+(\vert y-y_0\vert-R+1) q^u(y)\nu^u(\cdot,y),\;\text{ for }\;y\in \mathcal{B}_R(y_0)\setminus \mathcal{B}_{R-1}(y_0).
\end{array}\right.
\end{eqnarray}
From Lemma \ref{l-infty} we have
\[J_{C_R(y_0)}(u)\leq J_{C_R(y_0)}(v)=J_{\R\times (\mathcal{B}_R(y_0)\setminus \mathcal{B}_{R-1}(y_0))}(v),\]
and via (\ref{grad-bound-1})
\[J_{\R\times (\mathcal{B}_R(y_0)\setminus \mathcal{B}_{R-1}(y_0))}(v)\leq C\mathcal{L}^{n-1}(\mathcal{B}_R(y_0)\setminus \mathcal{B}_{R-1}(y_0))\leq C R^{n-2}.\]
  The proof of the lemma is complete.
\end{proof}
 \vskip1.1cm
 \noindent 3.\underline{The Isoperimetric Inequality for Minimizers}

As in the proof of the case $\alpha=2$ in Theorem A we let $\varphi:\mathcal{B}_R\subset\R^{n-1}\rightarrow\R$ be the solution of the problem
 \begin{eqnarray}\label{2phi}
 \left\{\begin{array}{l}\Delta\varphi=c_1\varphi,\;\text{ on }\;\mathcal{B}_R,\\
 \varphi=1,\;\text{ on }\;\partial \mathcal{B}_R,\end{array}\right.
 \end{eqnarray}
 where $c_1<c_0$ will be chosen later and $c_0$ is the constant in Lemma \ref{2w-properties}. We set
  \[q_M=\sup_{y\in\R^{n-1}}\|u(\cdot,y)\|\]
   and
 define
 \begin{equation}\label{2h-sigma-def}
 \begin{split}
 &h=e+q^h\nu^u,\quad q^h=\varphi q_M,\;\text{ and as before }\\
 &\sigma=e+q^\sigma\nu^u,\quad q^\sigma=\min\{q^u,q^h\},\\
 & \beta=\min\{q^u-q^\sigma,\lambda\},
 \end{split}
 \end{equation}
 where $\lambda\in(0,\bar{q})$ with $\bar{q}$ as in Lemma \ref{2w-properties}. We also recall the exponential estimate
 \begin{equation}\label{2exp-phi}
 \varphi(R-r)\leq e^{-c_2(R-r)},\;\text{ for }\;r\in[0,R],\;R\geq 1,
 \end{equation}
 for some $c_2>0$.

 We remark that the definition of $\sigma$ in (\ref{2h-sigma-def}) implies
 \[q^\sigma=q^u,\;\text{ on }\;\partial \mathcal{B}_R\]
 and that $\sigma\in W^{1,2}(C_R;\R^m)\cap L^\infty(C_R;\R^m)$ (see \cite{afs}).

Proceeding as in the proof of Theorem A by applying the inequality in \cite{eg} on $\mathcal{B}_R\subset\R^{n-1}$
to $\beta^2$ yields
\begin{equation}
\begin{split}
\Big(\int_{\mathcal{B}_R}\beta^\frac{2(n-1)}{n-2}dy\Big)^\frac{n-2}{n-1}&
=\Big(\int_{\mathcal{B}_R}(\beta^2)^\frac{n-1}{n-2}dy\Big)^\frac{n-2}{n-1}dy\\
&\leq C\int_{\mathcal{B}_R}\vert\nabla(\beta^2)\vert dy\;\;\quad(\beta=0,\;\text{ on }\;\partial B_R)\\
&\leq 2 C\int_{\mathcal{B}_R\cap\{q^u-q^\sigma\leq\lambda\}}\vert\nabla\beta\vert\vert\beta\vert dy\\
&\leq
 C A\int_{\mathcal{B}_R}\vert\nabla(q^u-q^\sigma)\vert^2 dy+
\frac{C}{A}\int_{\mathcal{B}_R\cap\{q^u-q^\sigma\leq\lambda\}}(q^u-q^\sigma)^2 dy\\&=
 C A\Big(\int_{\mathcal{B}_R}(\vert\nabla q^u\vert^2-\vert\nabla q^\sigma\vert^2) dy-
 2\int_{\mathcal{B}_R}\nabla q^\sigma\cdot\nabla(q^u-q^\sigma)dy\Big)\\&+
\frac{C}{A}\int_{\mathcal{B}_R\cap\{q^u-q^\sigma\leq\lambda\}}(q^u-q^\sigma)^2 dy
\end{split}\label{2iso-inequality}
\end{equation}
where we have utilized $\nabla\beta=0$ a.e. on $q^u-q^\sigma>\lambda$ and Young's inequality. Thus via (\ref{2young-1}) we derive
\begin{equation}
\begin{split}
&\Big(\int_{\mathcal{B}_R}\beta^\frac{2{n-1}}{n-2}dy\Big)^\frac{n-2}{n-1}\leq\\
 &\leq
 2C A\Big(\int_{\mathcal{B}_R}(\mathcal{W}(\sigma)-\mathcal{W}(u))dy-\int_{\mathcal{B}_R}\nabla q^\sigma\cdot\nabla(q^u-q^\sigma) dy\Big)\\&+
\frac{C}{A}\int_{\mathcal{B}_R\cap\{q^u-q^\sigma\leq\lambda\}}(q^u-q^\sigma)^2 dy
\end{split}
\label{2young}
\end{equation}
\vskip.8cm
\noindent 4.\underline{Conclusion}

The inequality (\ref{2young}), aside from the fact that $n$ is replaced by $n-1$, $\mathcal{B}_R$ is the ball of radius $R$ in $\R^{n-1}$ and $W$ is replaced by $\mathcal{W}$, coincides  with (\ref{young}). Moreover, by Lemma \ref{2w-properties},   $\mathcal{W}$ has the properties of $W$ in ({HA}), $\alpha=2$ and  Lemma \ref{2energy-upper-bound} is the counterpart of Lemma \ref{energy-upper-bound}. The only difference is that the inequality
\[W(h)-W(u)\leq W(h)\]
now is replaced by
\[\mathcal{W}(h)-\mathcal{W}(u)\leq \mathcal{W}(h).\]
Thus the arguments developed in the proof of Theorem A for the case $\alpha=2$ can be repeated verbatim to complete  the proof of Theorem B.
\subsection{The Proof of Theorem \ref{main-4}}
\noindent 1. First we note that under the hypotheses of Theorem \ref{main-4} we can take $\lambda^*=\|e_+-e_-\|$ in the statement of Theorem B. To argue this we let $\hat{\lambda}\in(0,\|e_+-e_-\|)$ and assume that
\[\mathcal{L}^{n-1}(\mathcal{B}_1(y_0)\cap
\{y:\|u(\cdot,y)-e_+(\cdot)\|\geq\hat{\lambda}\})\geq\mu_0>0.\]
Thus for $\lambda<\hat{\lambda}$, $\lambda>0$ as in Theorem B and fixed
\[\mathcal{L}^{n-1}(\mathcal{B}_1(y_0)\cap
\{y:\|u(\cdot,y)-e_+(\cdot)\|\geq\lambda\})\geq\mu_0>0.\]
Therefore
\[\mathcal{L}^{n-1}(\mathcal{B}_R(y_0)\cap
\{y:\|u(\cdot,y)-e_+(\cdot)\|\geq\lambda\})\geq C R^{n-1},\;R\geq 1.\]
We will be done if we can show that
\begin{equation}\label{willbedone}
\mathcal{L}^{n-1}(\mathcal{B}_R(y_0)\cap
\{y:\lambda\leq\|u(\cdot,y)-e_+(\cdot)\|\leq\hat{\lambda}\})\leq C R^{n-2}.
\end{equation}
For this purpose note that $\lambda\leq\|u(\cdot,y)-e_+(\cdot)\|\leq\hat{\lambda}$ implies
\[\|u(\cdot,y)-e_-(\cdot)\|\geq\|e_+-e_-\|-\|u(\cdot,y)-e_+(\cdot)\|
\geq\|e_+-e_-\|-\hat{\lambda}=\tilde{\lambda}>0.\]
Thus
on $S_\lambda^{\hat{\lambda}}=\mathcal{B}_R(y_0)\cap
\{y:\lambda\leq\|u(\cdot,y)-e_+(\cdot)\|\leq\hat{\lambda}\}$ we have the estimate
\[\mathcal{W}(u(\cdot,y))\geq \bar{w}(\tilde{\lambda})>0\]
and so via (\ref{2energy-upper-bound-1})
\[\bar{w}(\tilde{\lambda})\mathcal{L}^{n-1}(S_\lambda^{\hat{\lambda}})\leq
\int_{S_\lambda^{\hat{\lambda}}}\mathcal{W}(u(\cdot,y))dy\leq C R^{n-2},\]
and so (\ref{willbedone}) is established.

\noindent 2. Suppose now that
\[\mathcal{L}^{n-1}(\mathcal{B}_1(y_0)\cap
\{y:\|u(\cdot,y)-e_+(\cdot)\|\leq\theta\})>\mu_0>0.\]
Since $\lambda=\|e_+-e_-\|-\theta$ implies $\{y:\|u(\cdot,y)-e_-(\cdot)\|\leq\theta\}\subset
\{y:\|u(\cdot,y)-e_-(\cdot)\|\geq\lambda\}$
it follows that
\[\mathcal{L}^{n-1}(\mathcal{B}_1(y_0)\cap
\{y:\|u(\cdot,y)-e_-(\cdot)\|>\lambda\})\geq\mu_0>0.\]
Hence by 1. above
\[\mathcal{L}^{n-1}(\mathcal{B}_R(y_0)\cap
\{y:\|u(\cdot,y)-e_-(\cdot)\|>\lambda\})\geq C R^{n-1}.\;R\geq 1.\]
From this it easily follows that
\[\mathcal{L}^{n-1}(\mathcal{B}_R(y_0)\cap
\{y:\|u(\cdot,y)-e_+(\cdot)\|\leq\theta\})\geq C R^{n-1}.\]
The proof of Theorem \ref{main-4} is complete.
\subsection{On the Product Structure of Solutions}
In this subsection we give alternative proofs of some of the results in \cite{af4}.
\begin{theorem}\label{teo-af}{\rm (\cite{af4}, Theorem 1.2)} Assume that $W$ satisfies {\rm ({ha})} and {\rm ({hb})} and assume that the connection $e$ in {\rm ({hb})} is unique. Let $O\subset\R^{n-1}$, $O\neq\R^{n-1}$ be open with $\,\sup_{y_0\in O}R_{y_0}=+\infty$ {\rm(}$R_{y_0}=\sup_R\{\mathcal{B}_R(y_0)\subset O\}${\rm)} and assume that $u:\R\times O\rightarrow\R^m$ is minimal in the sense of {\rm ({hc})} {\rm(}with $\R^n$ replaced by $\R\times O${\rm)}. Then there are constants $k_0, K_0>0$ such that
\[\vert u(s,y)-e(s)\vert\leq K_0e^{-k_0 d(y,\partial O)}.\]
\end{theorem}
\begin{proof} It is sufficient to establish
 that, given a small number $\gamma>0$, there is $d_\gamma>0$ such that
\begin{equation}\label{d}
d(y,\partial O)\geq d_\gamma\quad\Rightarrow\quad\vert u(s,y)-e(s)\vert<\gamma .
\end{equation}
since then linear theory renders the result.

By Lemma \ref{linf-l2} below  there is a constant $C>0$ such that
\[\|u(\cdot,y)-e(\cdot\,)\|_{L^\infty(\R;\R^m)}\leq C\|u(\cdot,y)-e(\cdot\,)\|^\frac{2}{3}.\] Therefore $\vert u(s,y)-e(s)\vert\geq\gamma$ implies
\begin{equation}\label{w}
\|u(\cdot,y)-e(\cdot\,)\|\geq(\frac{\gamma}{C})^\frac{3}{2}.
\end{equation}
From the assumed uniqueness and hyperbolicity of $e$ it follows that, given $\lambda>0$ small, it results
\[\|u(\cdot,y)-e(\cdot\,)\|\geq\lambda\quad\Rightarrow\quad\mathcal{W}(u(\cdot,y))\geq \bar{w}(\lambda)>0.\]
Therefore arguing as in the proof of Theorem \ref{teo-f} we deduce from Theorem B and Lemma \ref{2energy-upper-bound} that there is $R(\lambda)>0$ such that
\begin{equation}\label{righarrow}
\mathcal{B}_{R(\lambda)}(y_0)\subset O\quad\Rightarrow\quad\|u(\cdot,y_0)-e(\cdot\,)\|<\lambda
\end{equation}
This and (\ref{w}) imply that we can take $d_\gamma=R(\frac{\gamma}{C})^\frac{3}{2})$ in (\ref{d}). The proof is complete.
\end{proof}
We now establish
\begin{lemma}\label{linf-l2}{\rm(cfr. \cite{af4}  Lemma 2.2)} Let $v\in\mathrm{E}^{\mathrm{xp}}$. Then
\begin{equation}\label{9}\|v\|_{L^\infty(\R;\R^m)}\leq C\|v\|^\frac{2}{3},
\end{equation}
where $C=C(k,K)>0$ {\rm((\ref{grad-bound-1}))} is independent of $v$.
\end{lemma}
\begin{proof}
\[\begin{split}\vert v(s)\vert^p&=\int_{-\infty}^s\frac{\partial}{\partial t}\vert v(t)\vert^p dt\leq p\int_\R\vert v(t)\vert^{p-1}\vert v_t(t)\vert dt\\&\leq p
\int_\R(\vert v(t)\vert^{p^\prime(p-1)})^\frac{1}{p^\prime}(\int_\R\vert v_t(t)\vert^{q^\prime} dt)^\frac{1}{q^\prime}
\quad\quad(\frac{1}{p^\prime}+\frac{1}{q^\prime}=1)\\&
\leq C\|v\|_{L^{p^\prime(p-1)}(\R;\R^m)}^{p-1}.
\end{split}\]
Hence
\[\|v\|_{L^\infty(\R;\R^m)}\leq C\|v\|_{L^{p^\prime(p-1)}(\R;\R^m)}^\frac{p-1}{p}.\]
Choosing first $p^\prime$ so that $p^\prime(p-1)=2$ and finally noting that $\max\frac{p-1}{p}=\frac{2}{3}$ we arrive at (\ref{9}). The proof of the lemma is complete.
\end{proof}

In \cite{af4} Theorem \ref{teo-af} was established by a different approach which also applies to a larger class of minimizers not necessarily defined on cylinders. We conclude with the following  {\it Rigidity} result
\begin{theorem}{\rm (see Theorem 1.3 in \cite{af4})} Assume $u:\R^n\rightarrow\R^m$ and otherwise the hypothesis of Theorem \ref{teo-af}. Then
\[u(x)=e(x_1),\;\text{ for }\;x\in\R^n.\]
\end{theorem}
\begin{proof} The argument is essentially the same as in the proof of Theorem \ref{rigidity}. For each $y_0\in\R^{n-1}$ and for each $\lambda>0$ we have trivially $\mathcal{B}_{R(\lambda)}(y_0)\subset\R^{n-1}$ and therefore, using also Lemma \ref{linf-l2}
\[C^{-\frac{3}{2}}(\|u(\cdot,y_0)-e(\cdot\,)\|_{L^\infty(\R;\R^m)})^\frac{3}{2}\leq \|u(\cdot,y_0)-e(\cdot\,)\|<\lambda,\;\text{ for each }\;y_0\in\R^{n-1},\,\lambda>0.\]
The proof is concluded.

\end{proof}
\bibliographystyle{plain}

\begin{thebibliography}{99}
\bibitem{abg}
S.~Alama, L.~Bronsard, and C.~Gui.
 \newblock Stationary solutions in R2 for an Allen-Cahn system with multiple well potential,
 \newblock {\em Calc.\ Var.\ Part.\ Diff.\ Eqs.} {\bf 5} No.~4 (1997), pp.~359--390.


\bibitem{a}
N.~D.~Alikakos.
\newblock  Some basic facts on the system $\Delta u-W_u(u)=0$.
\newblock {\em Proc.\ Amer.\ Math.\ Soc.} {\bf 139} No.~1 (2011), pp.~153--162.
\bibitem{a2}
N.~D.~Alikakos.
\newblock On the structure of phase transition maps for three or more coexisting phases.
\newblock In {\em  Geometric partial differential equations, {\rm M. Novaga and G. Orlandi} eds. Publications  Scuola Normale Superiore}, CRM Series, Birkhauser, (2013).
\bibitem{a3}
N.~D.~Alikakos.
\newblock  A new proof for the existence of an equivariant entire solution connecting the minima of the potential for the system $\Delta u-W_u(u)=0$.
\newblock {\em
Comm.\ Partial\ Diff.\ Eqs} {\bf 37} No.~12 (2012) pp.~2093�2115.
\bibitem{af2}
N.~D.~Alikakos and G.~Fusco.
\newblock Entire solutions to equivariant elliptic systems with variational structure.
\newblock {\em Arch.\ Rat.\ Mech.\ Analysis} {\bf 202} No.~2 (2011), pp.~567--597.
\bibitem{af}
N.~D.~Alikakos and G.~Fusco.
\newblock
\newblock {\em in preparation}
\bibitem{af3}
N.~D.~Alikakos and G.~Fusco.
\newblock On the connection problem for potentials with several global minima.
\newblock {\em Indiana\ Univ. \ Math. \ Journal} {\bf 57} (2008), pp.~1871--1906.
\bibitem{af4}
N.~D.~Alikakos and G.~Fusco
\newblock Asymptotic and rigidity results for symmetric solutions of the elliptic system $\Delta u=W_u(u)$.
\newblock arXiv:1402.5085.
\bibitem{afs}
N.~D.~Alikakos, G.~Fusco and P.~Smyrnelis
\newblock {\em .Monograph} (in preparation).
\newblock
\bibitem{b}
S.~Baldo.
\newblock Minimal interface criterion for phase transitions in mixtures of Cahn-Hilliard fluids.
\newblock {\em Ann.\ Inst. \ Henri\ Poincare} {\bf 7} No.~2 (1990), pp.~67--90.

\bibitem{bbh}
F.~Bethuel, H.~Brezis and F.~Helein.
\newblock Ginzburg-Landau Vortices.
\newblock Birkh�user (1994).

\bibitem{br}
L.~Bronsard and F.~Reitich.
\newblock On three-phase boundary motion and the singular limit of a vector-valued Ginzburg-Landau equation.
\newblock {\em Arch.\ Rat.\ Mech.\ Analysis} {\bf 124} No.~4 (1993), pp.~355--379.

\bibitem{bgs}
L.~Bronsard, C.~Gui, and M.~Schatzman.
\newblock A three-layered minimizer in $\R^2$ for a var\-i\-ational problem with a symmetric three-well potential.
\newblock {\em Comm.\ Pure.\ Appl.\ Math.} {\bf 49} No.~7 (1996), pp.~677--715.
\bibitem{cc}
L.~Caffarelli and A.~Cordoba.
\newblock Uniform convergence of a singular perturbation problem.
\newblock {\em Comm.\ Pure \ Appl.\ Math.} {\bf 48} No.~ (1995), pp.~1--12.

\bibitem{cl}
L.~Caffarelli and F.~Lin.
\newblock Singularly perturbed elliptic systems and multi-valued harmonic functions
with free boundaries.
\newblock {\em Journal\ Amer.\ Math.\ Society
} {\bf 21} (2008), pp.~847--862.


\bibitem{eg}
L.~C.~Evans and R.~F.~Gariepy
\newblock {\em Measure Theory and Fine Properties of Functions.}
\newblock  Studies in Advanced Mathematics.\, CRC Press New York \,(1992).


\bibitem{fa}
A.~Farina.
\newblock Two results on entire solutions of Ginzburg-Landau systems in higher dimensions.
\newblock {\em J.\ Funct. \ Anal.} {\bf 214} No.~2 (2004), pp.~386--395.
\bibitem{fv}
A.~Farina and E.~Valdinoci.
\newblock Geometry of quasiminimal phase transitions.
\newblock {\em Calc.\ Var.\ Part.\ Diff.\ Eqs.} {\bf 33} No.~ (2008), pp.~1--35.



\bibitem{fu1}
\newblock G.~Fusco.
\newblock Equivariant entire solutions to the elliptic system $\Delta u=W_u(u)$ for general $G-$invariant potentials.
\newblock {\em Calc.\ Var.\ Part.\ Diff.\ Eqs.} February (2013), pp.~1--23.



\bibitem{fu}
G.~Fusco.
\newblock On some elementary properties of vector minimizers of the Allen-Cahn energy.
\newblock  {\em \ Comm.\ Pure \ Appl. \ Anal. } {\bf 13} No.~3 (2014),  pp.~1045--1060.

\bibitem{gmt1}
E.~Gonzalez, U.~Massari and I.~Tamanini.
\newblock On the regularity of boundaries of sets minimizing perimeter with a volume constraint.
\newblock {\em Indiana\ Univ. \ Math. \ Journal} {\bf 32} (1983), pp.~25--37.


\bibitem{gs}
C.~Gui and M.~Schatzman.
\newblock Symmetric quadruple phase transitions.
\newblock {\em Ind.\ Univ.\ Math.\ J.} {\bf 57} No.~2 (2008), pp.~781--83

\bibitem{sa}
O.~Savin.
\newblock Minimal Surfaces and Minimizers of the Ginzburg Landau energy.
\newblock  {\em \ Cont.\ Math. \ AMS } {\bf 526} (2010),  pp.~43--58.

\bibitem{sv}
O.~Savin and E.~Valdinoci.
\newblock Density estimates for a variational model driven bt the Gagliardo norm.
\newblock  arXiv:1007.2114.

\bibitem{sv2}
O.~Savin and E.~Valdinoci.
\newblock Density estimates for a nonolocal variational model via the Sobolev inequality.
\newblock  arXiv:1103.6205.

\bibitem{siv}
Y.~Sire and E.~Valdinoci.
\newblock Density estimates for phase transitions with a trace.
\newblock   arXiv: 1011.6617.


\bibitem{s}
P.~Smyrnelis.
\newblock
\newblock {\em personal comunication. } {\bf }


\bibitem{v}
E.~Valdinoci.
\newblock Plane-like minimizers in periodic media: jet flows and
Ginzburg-Landau-type functionals.
\newblock {\em J.\ Reine\  Angew.\ Math. } {\bf 574} (2004), pp.~147--185.
\bibitem{gmt}
B.~White
\newblock {\em Topics in GMT.}
\newblock  Notes by O.Chodash.\, Stanford \,(2012).

\end{thebibliography}

\vskip.2cm
(N.D. ALIKAKOS) Department of Mathematics, University of Athens, Panepistemiopolis, 15784 Athens, Greece; e-mail: {\texttt{nalikako@math.uoa.gr}}
\vskip.2cm
\noindent (G. FUSCO) Universit\`a degli Studi dell'Aquila, Via Vetoio, 67010 Coppito, L'Aquila, Italy; e-mail:{\texttt{fusco@univaq.it}}

\end{document}